\documentclass[smallextended,numbook,runningheads]{svjour3}
\usepackage{amsmath}
\smartqed  

\usepackage{amsfonts,amssymb}
\usepackage{epsfig}
\usepackage{booktabs}
\usepackage{xcolor}

\journalname{}
\date{ \phantom{b} \vspace{45mm}\phantom{e}}
\def\makeheadbox{\hspace{-4mm}Version of 22 January 2021}

\def\bigo{{\mathcal O}}
\def\real{{\mathbf R}}

\def\d{{\mathrm d}}
\def\e{{\mathrm e}}

\def\iu{\mathrm{i}}
\def\eps{\varepsilon}

\DeclareMathOperator{\sinc}{sinc}
\DeclareMathOperator{\tanc}{tanc}
\DeclareMathOperator{\sinch}{sinch}
\DeclareMathOperator{\tanch}{tanch}

\def\Re{{\mathrm{Re}\,}}

\newdimen\GGGlength
\newdimen\GGGheight
\newbox\GGGbox
\def\GGGput[#1,#2](#3,#4)#5{%
  \setbox\GGGbox\vbox{\hbox{#5}\kern0pt}%
  \GGGlength\wd\GGGbox%
  \divide\GGGlength by100 \multiply\GGGlength by#1%
  \GGGheight\ht\GGGbox%
  \divide\GGGheight by100 \multiply\GGGheight by#2%
  \put(#3,#4){\kern-\GGGlength\raise-\GGGheight\box\GGGbox}}

\begin{document}

\title{Large-stepsize integrators for charged-particle dynamics over multiple time scales}

\titlerunning{Charged-particle dynamics over multiple time scales}

\author{Ernst Hairer$^1$, Christian Lubich$^2$, Yanyan~Shi$^3$}
\authorrunning{E.\ Hairer, Ch.\ Lubich, Y.\ Shi}

\institute{$^1$~Dept.\ de Math{\'e}matiques, Univ.\ de Gen{\`e}ve,
CH-1211 Gen{\`e}ve 24, Switzerland.\\
\phantom{$^1$~}\email{Ernst.Hairer@unige.ch}\\
$^2$~Mathematisches Institut, Univ.\ T\"ubingen, D-72076 T\"ubingen, Germany.\\
\phantom{$^2$~}\email{Lubich@na.uni-tuebingen.de}\\
$^3$~LSEC, Academy of Mathematics and Systems Science, Chinese Academy of Sciences,\\ 
\phantom{$^3$~}Beijing 100190, China; University of Chinese Academy of Sciences, Beijing 100049, China.\\
\phantom{$^3$~}\email{shiyanyan1995@lsec.cc.ac.cn}
}

\date{ }

\maketitle

\begin{abstract} The Boris algorithm, a closely related variational integrator and a newly proposed filtered variational integrator are studied when they are used to numerically integrate the equations of motion of a charged particle in a non-uniform strong magnetic field,
taking step sizes that are much larger than the period of the Larmor rotations. For the Boris algorithm and the standard (unfiltered) variational integrator, satisfactory behaviour is only obtained when the component  of the initial velocity orthogonal to the magnetic field is filtered out. The particle motion shows varying behaviour over multiple time scales: fast Larmor rotation, guiding centre motion, slow perpendicular drift, near-conservation of the magnetic moment over very long times and conservation of energy for all times. Using modulated Fourier expansions of the exact and numerical solutions, it is analysed  to which extent this behaviour is reproduced by the three numerical integrators used with large step sizes.

\bigskip

\end{abstract}

\section{Introduction}\label{sec:intro}

The time integration of the equations of motion of  charged particles is a basic algorithmic task for particle methods in plasma physics \cite{birdsall05ppv}.
In this paper we consider the case of a non-uniform strong magnetic field in the asymptotic scaling known as maximal ordering \cite{brizard07fon,possanner18gfv}, with a small parameter $\eps \ll 1$  whose inverse corresponds to the strength of the magnetic field.
The particle motion then shows different behaviour over multiple time scales: 
\begin{itemize}
\item  fast Larmor rotation over the time scale $\eps$,  
\item guiding centre motion over the time scale $\eps^0$, 
\item slow drift perpendicular to the magnetic field over the time scale $\eps^{-1}$, 
\item near-conservation of the magnetic moment over time scales $\eps^{-N}$ with arbitrary $N>1$, 
\item and energy conservation for all times.
\end{itemize}

In this paper we are interested in using numerical integrators with step sizes $h$ that are much larger than the quasi-period $2\pi \eps$ of the Larmor rotation. 
We thus have the two small parameters $h$ and $\eps$, which we will assume to be related by
\begin{equation}\label{eps-h2}
0 < \eps \le h^2 \ll 1.
\end{equation}
We study the behaviour of the numerical integrators over the time scales $\eps^0$, $\eps^{-1}$, and $\eps^{-N}$ for $N>1$.
We are not aware of previous numerical analysis in this large-stepsize regime.  With an emphasis on different aspects, recent papers on numerical methods for charged-particle dynamics in a strong magnetic field include \cite{chartier19uam,chartier20uam,crouseilles17uap,filbet16asp,filbet17app,filbet20cao,hairer20lta,hairer20afb,ricketson20aec,wang20eeo}.

In Section 2 we formulate the equations of motion in the scaling considered here and illustrate the solution behaviour over various time scales.

In Section 3 we describe the three numerical integrators studied in this paper: the {\it Boris algorithm} \cite{boris70rps,qin13wib,ellison15cos,hairer18ebo}, a closely related {\it variational integrator} \cite{webb14sio,hairer20lta}, and a newly proposed {\it filtered variational integrator}, which only requires a minor algorithmic modification of the standard variational integrator and can be interpreted as the standard variational integrator for a Lagrangian with an anisotropically modified kinetic energy term.

In Section 4 we give modulated Fourier expansions of the exact solution and of the numerical solutions of the three numerical methods used with step sizes \eqref{eps-h2}. The differential equations for the dominant modulation functions are the key to understanding the method behaviour over the times scales $\eps^0$ and $\eps^{-1}$ for all three methods.
For the Boris algorithm and the standard (unfiltered) variational integrator, the initial velocity needs to be modified such that its component perpendicular to the magnetic field is $O(\eps)$-small. The complete modulated Fourier expansion will be used for studying the long-time near-conservation of the magnetic moment and energy for the filtered variational integrator.

In Section 5 we obtain $O(h^2)$ error bounds uniformly in $\eps$ for all three (formally second-order) numerical methods over the time scale $\eps^0$. This is not an obvious result for large step sizes \eqref{eps-h2} but here it follows directly from a comparison of the modulated Fourier expansions of the exact and numerical solutions. 

In Section 6 we show that all three methods reproduce the perpendicular drift with an $O(h^2)$ or $O(h)$ error over the time scale $\eps^{-1}$. This is again obtained via the modulated Fourier expansions, which also yield an $O(\eps)$ approximation to the perpendicular drift by the solution of a slow differential equation over times $O(\eps^{-1})$.

In Section 7  we consider the long-term energy behaviour. For the standard variational integrator with the modified starting velocity we prove
near-conservation of the total energy up to time $O(\eps^{-1})$. For the filtered variational integrator we prove near-conservation of magnetic moment and energy over times $\eps^{-N}$ with arbitrary $N>1$ for non-resonant step sizes, using the Lagrangian structure of the modulation system. Moreover, we show results of numerical experiments for the energy behaviour of the three methods over long times.

The conclusion of our investigation is that the new filtered variational integrator with non-resonant large step sizes  \eqref{eps-h2} reproduces the characteristic features well over all time scales, and this is fully explained by our theory. The Boris algorithm and the standard (unfiltered) variational integrator also work remarkably well for large stepsizes \eqref{eps-h2} on the time scales $\eps^0$ and $\eps^{-1}$ in accordance with our theory, provided that the initial velocity is modified such that the component perpendicular to the magnetic field is reduced to size $O(\eps)$. With this filtering of the starting velocity, the long-time energy behaviour of the Boris method and the standard variational integrator  appears to be better  in our numerical experiments than we can explain by theory.

\section{Multiple time scales in the continuous problem}

We study the time integration of the equations of motion of a charged particle in a strong magnetic field,
with position $x(t)\in\real^3$ and velocity $v(t)=\dot x(t)$ at time $t$,
\begin{equation}\label{ode}
\begin{aligned}
&\ddot x(t) =  \dot x(t) \times B(x(t)) + E(x(t))
\\
&\text{with}\quad
B(x) = \frac 1 \eps \, B_0 + B_1(x) \quad\text{for }\  0<\eps\ll 1 ,
\end{aligned}
\end{equation}
where $B_0$ is a fixed vector in $\real^3$ of unit norm, $|B_0|=1$. The non-constant magnetic field $B_1(x)$  is assumed to have a known vector potential $A_1(x)$, i.e. $B_1(x)= \nabla_x \times A_1(x)$. This gives
$B(x)= \nabla_x \times A(x)$ with the vector potential $A(x) = -\frac12 x \times B_0/\eps + A_1(x)$. 
We always assume that $B_1:\real^3\to\real^3$ and $E:\real^3\to\real^3$ are smooth with derivatives bounded independently of $\eps$ on bounded subsets of $\real^3$. The above scaling corresponds to what is known as {\it maximal ordering} in the literature; see \cite{brizard07fon,possanner18gfv}.

For the initial position and velocity we always assume boundedness independently of $\eps$:
\begin{equation}\label{init}
|x(0)| \le C_0, \quad\ |\dot x(0)| \le C_1.
\end{equation}

For studying the perpendicular drift, we need further assumptions on $B_1$ and $E$ that are specified in Section 6.
When it comes to studying the long-time energy behaviour, we further assume that the force field has a scalar potential, $E(x)=-\nabla \phi(x)$. The total energy is then
\begin{equation}\label{H}
H(x,v) = \tfrac12 |v|^2 + \phi(x),
\end{equation}
which is conserved along every trajectory and is bounded independently of~$\eps$ under condition \eqref{init}.
We further consider the magnetic moment (rescaled with $\eps$), 
\begin{equation}\label{I}
I(x,v)= \frac1{2\eps} \frac{|v \times B(x)|^2}{|B(x)|^3}.
\end{equation}
We note that, with $v_\perp$ denoting the velocity component orthogonal to $B(x)$, 
\begin{equation}\label{I-vperp}
I(x,v)= \frac12\,\frac{|v \times B_0+ O(\eps)|^2}{1+O(\eps)}  =  \tfrac12|v_\perp|^2 (1+O(\eps))+ O(\eps^2),
\end{equation}
for $(x,v)$ in any region that is bounded independently of $\eps$. The magnetic moment is an adiabatic invariant: it is conserved up to $O(\eps)$ over  very long times $t\le \eps^{-N}$ with arbitrary $N>1$;
see e.g. \cite{kruskal58tgo,northrop63tam,benettin94aia,hairer20lta}.

\begin{figure}[!ht]
\label{fig:timescales}
\centerline{
\includegraphics[scale=0.48]{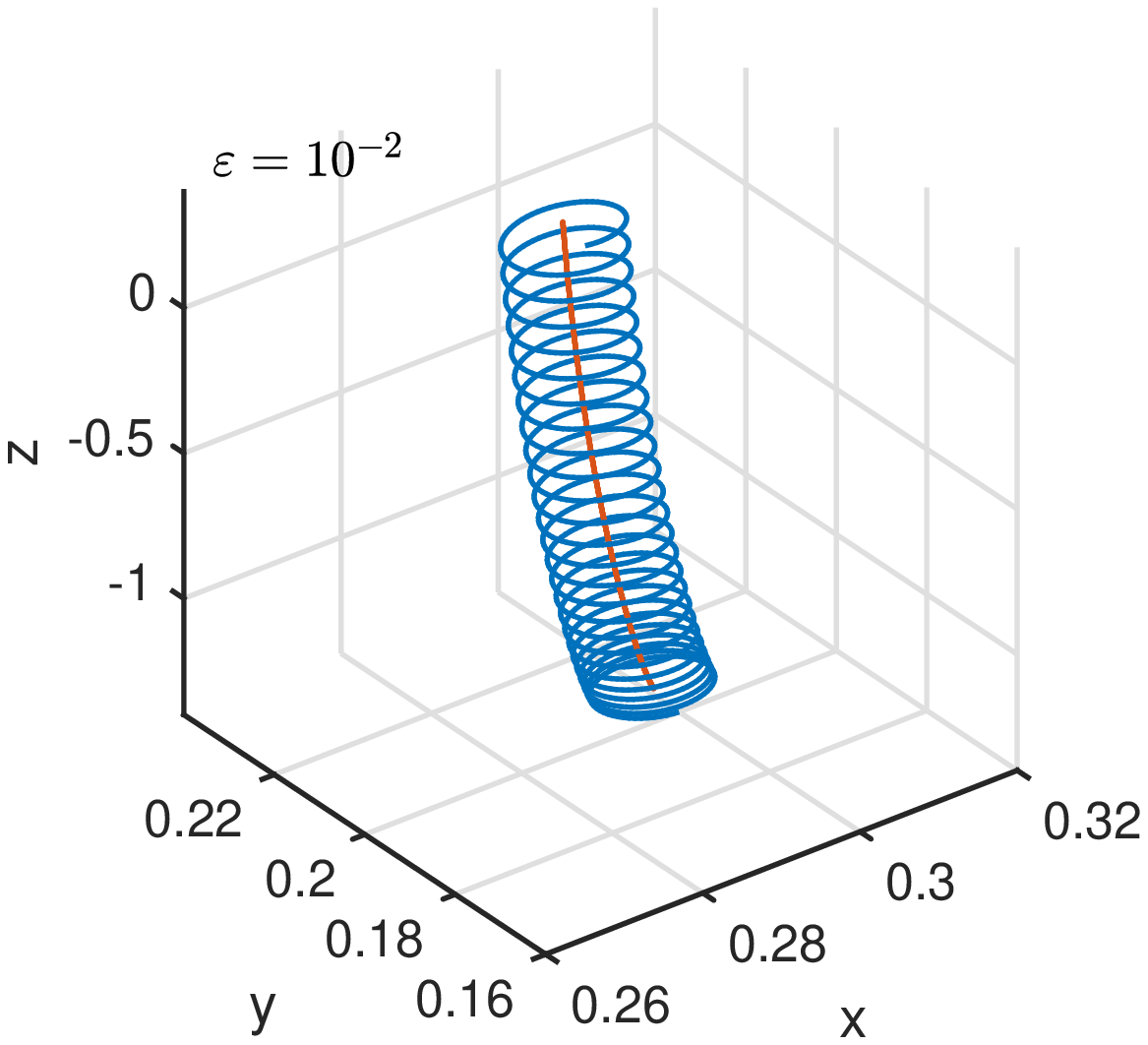}
\includegraphics[scale=0.48]{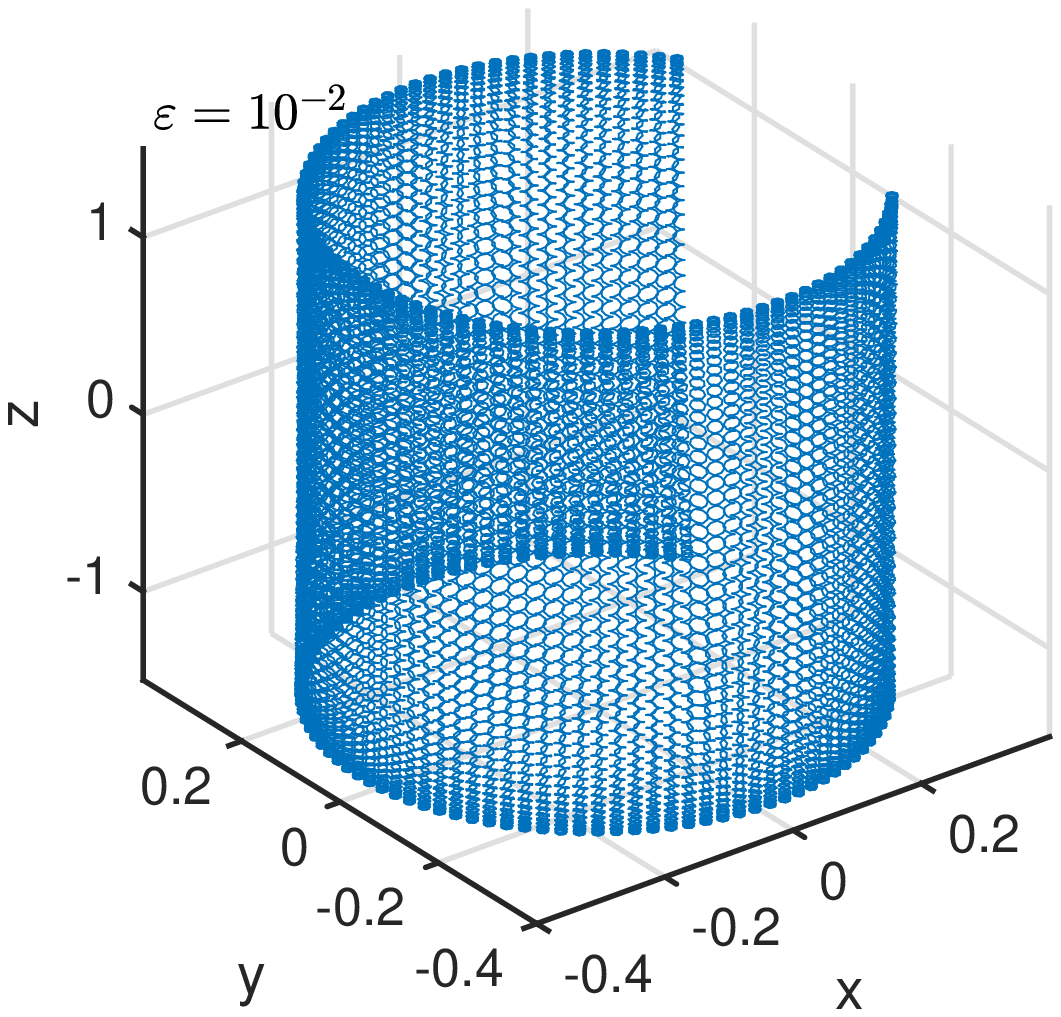}}
\centerline{\includegraphics[scale=0.5]{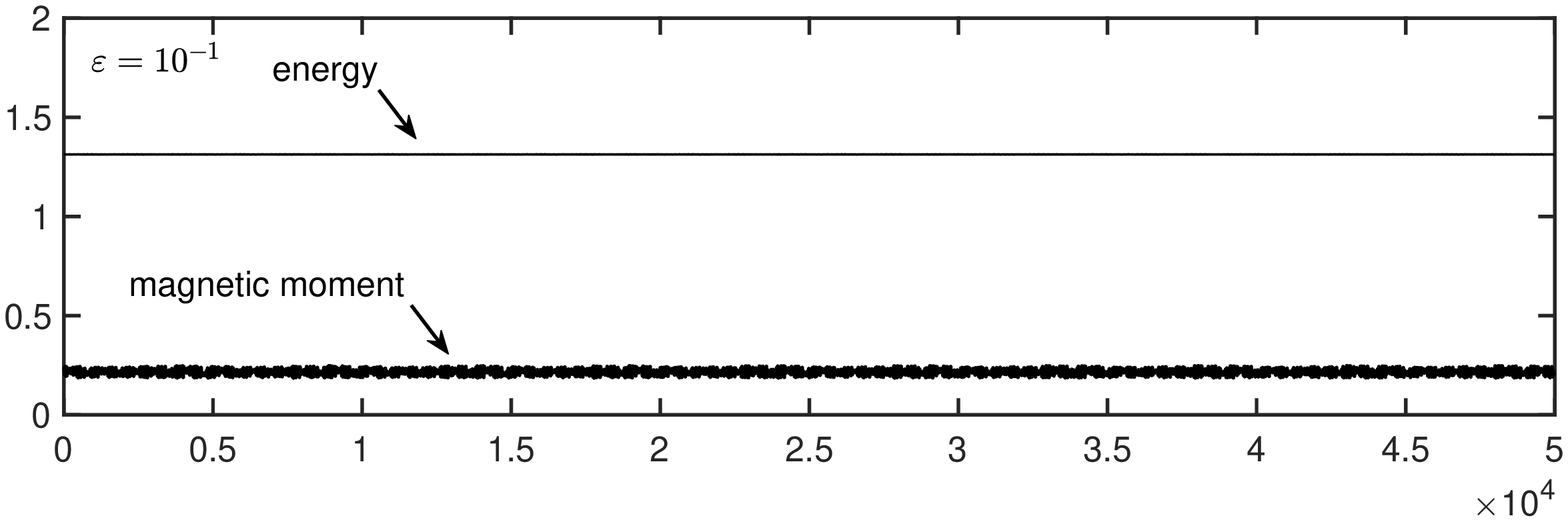}}
\caption{Trajectories of the particle for $t\le \pi/2$ (top left) and $t \le 5/\eps$ (top right).  Energy and magnetic moment for $t\le \eps^{-4}$ (bottom). The analogous picture for $\eps=10^{-2}$ would show the magnetic moment as a horizontal straight line.}
\end{figure}

In Figure~2.1 
we illustrate the solution behaviour on various time scales. We show the fast Larmor rotation of angular frequency $\eps^{-1}$ and amplitude $O(\eps)$ on the time scale $\eps$ and the guiding centre motion on the time scale $\eps^0$ in the first picture, and in addition the slow drift perpendicular to the magnetic field  on the time scale $\eps^{-1}$ in the second picture (here: horizontal drift for the magnetic field in vertical direction). Finally, the third picture shows the long-time near-conservation of the magnetic moment and the conservation of energy. Our objective is to understand how the behaviour on the various time scales can be replicated by numerical methods with large time steps that do not resolve the fast Larmor rotations.

In Figure~2.1 we take the electromagnetic fields and the vector and scalar potentials as
\begin{align*}
&B(x)=
\frac{1}{\varepsilon}
\begin{pmatrix}
0\\0\\1
\end{pmatrix}+
\begin{pmatrix}
x_1(x_3-x_2)\\ x_2(x_1-x_3) \\x_3(x_2-x_1)
\end{pmatrix}
\ 
\text{ with }
A(x) = -\frac1{2} 
\begin{pmatrix}
x_2 \\ -x_1 \\ 0
\end{pmatrix}
+ x_1x_2x_3 \begin{pmatrix}
1\\1\\1
\end{pmatrix},
\\[1mm]
&E(x) = -x\ \text{ with } \phi(x)=\tfrac12 |x|^2,
\end{align*}
and the initial values 
$x(0)=(0.3,0.2,-1.4)^\top$  and $\dot x(0)=(-0.7,0.08,0.2)^\top$.

\section{Three numerical integrators}

We now describe the three numerical integrators for \eqref{ode} that are studied in this paper when applied with large step sizes $h\gg \eps$.

\subsection{Boris algorithm}
The Boris method, introduced in \cite{boris70rps}, is the standard integrator for particle-in-cell codes for
plasma simulation; see e.g. \cite{birdsall05ppv,derouillat18sac}. Given the position and velocity approximation $(x^n,v^{n-1/2})$, the algorithm computes $(x^{n+1},v^{n+1/2})$ as follows, with  $B^n=B(x^n)$ and $E^n=E(x^n)$:
\begin{align} \nonumber
v_{+}^{n-\tfrac{1}{2}}&=v^{n-\tfrac{1}{2}}+\frac{h}{2}E^n
\\ \nonumber
v^{n+\tfrac{1}{2}}_{-}-v^{n-\tfrac{1}{2}}_{+}&=\frac{h}{2}(v^{n+\tfrac{1}{2}}_{-}+v^{n-\tfrac{1}{2}}_{+})\times B^n
\\ \nonumber
v^{n+\tfrac{1}{2}}&=v^{n+\tfrac{1}{2}}_{-}+\frac{h}{2}E^n
\\
x^{n+1}&=x^n+hv^{n+\tfrac{1}{2}},
\label{boris-1}
\end{align}
where the starting value is chosen as $v^{1/2}=v^0 + \frac{h}{2}v^0\times B^0 + \frac{h}{2}E^0$.

The method has the equivalent two-step formulation 
\begin{equation}\label{eq:Boris}
\frac{x^{n+1}-2x^n+x^{n-1}}{h^2}=\frac{x^{n+1}-x^{n-1}}{2h}\times B^n +E^n,
\end{equation}
with the velocity approximation
\begin{equation}\label{vn}
v^n=\frac{x^{n+1}-x^{n-1}}{2h}.
\end{equation}

It is known from \cite{ellison15cos} that the Boris algorithm is not symplectic unless $B$ is a constant magnetic field. 
The energy behaviour over long times, which is not fully satisfactory, has been studied in \cite{hairer18ebo} for step sizes with $h|B|\ll 1$, which in our case \eqref{ode} would read $h\ll\eps$ in contrast to \eqref{eps-h2}.

In the large-stepsize regime \eqref{eps-h2}  the starting velocity needs to be modified. Instead of setting $v^0$ equal to the initial data $\dot x(0)$ we choose $v^0$ such that its component $v^0_\perp$ orthogonal to the magnetic field is $O(\eps)$-small. We propose to take
$v^0=  v^0_\parallel+ v^0_\perp$ with
\begin{equation}\label{v-init-mod}
v^0_\parallel =P_0\dot x(0), \quad v^0_\perp = \eps\bigl(v^0_\parallel \times B_1(x^0)+E(x^0)\bigr)\times B_0,
\end{equation}
where $P_0=B_0B_0^\top$ is the orthogonal projection in the direction of $B_0$.
(This choice of $v^0_\perp$ will be explained in Section 4 right after Theorem~\ref{thm:mfe-boris}.)
Without such a modification of the starting velocity, the Boris algorithm shows highly oscillatory behaviour with a large amplitude proportional to $(h^2/\eps)|v^0_\perp|$; cf.~\cite{ricketson20aec}.

Since the Boris method with large step size \eqref{eps-h2} and the proposed filtering of the initial velocity will give an approximation to the guiding centre rather than to the oscillatory trajectory, it is reasonable to take the guiding centre approximation $x(0)+ \eps \dot x(0)\times B_0$ instead of $x(0)$ as the starting position $x^0$.

We note that while the one-step map
$(x^n,v^{n-1/2}) \mapsto (x^{n+1},v^{n+1/2})$
is volume-preserving \cite{qin13wib}, the starting-value map $(x(0),\dot x(0)) \mapsto (x^0,v^0)$ and also the map $(x^0,v^0)\mapsto (x^1,v^{1/2})$ are far from volume-preserving for step sizes \eqref{eps-h2}.

\subsection{Standard variational integrator} 
The variational integrator to be studied here is constructed in the same way as is done in the interpretation of the St\"ormer--Verlet method  as a variational integrator; see e.g. \cite[Chap.\,VI, Example 6.2]{hairer06gni} and \cite{webb14sio}.  The integral of the Lagrangian $L(x,v)=\tfrac12 |v|^2 + A(x)^\top v - \phi(x)$ over a time step is approximated in two steps: the path $x(t)$ of positions is approximated by the linear interpolant of the endpoint positions, and the integral is approximated by the trapezoidal rule. This approximation to the action integral is then extremized. With
the derivative matrix $A'(x)=(\partial_j A_i(x))_{i,j=1}^3$ and its transpose $A'(x)^\top$, this variational integrator becomes the following:
\begin{align}\label{varint}
 &\frac{x^{n+1}-2 x^n + x^{n-1}}{h^2} =
 \\
 \nonumber
 &\qquad A'(x^n)^\top \,\frac {x^{n+1}-x^{n-1}}{2h}
-  \frac{A(x^{n+1})-A(x^{n-1})}{2h} +E^n,
\end{align}
or equivalently, written as a perturbation to the Boris algorithm and using that $v\times B(x)=A'(x)^\top v - A'(x)v$,
\begin{align} \nonumber
 &\frac{x^{n+1}-2 x^n + x^{n-1}}{h^2} = \frac{x^{n+1}-x^{n-1}}{2h}\times B^n +E^n
 \\
\label{varint-borislike}
 &\qquad + A'(x^n) \,\frac {x^{n+1}-x^{n-1}}{2h}
-  \frac{A(x^{n+1})-A(x^{n-1})}{2h}\, .
\end{align}
We note that the correction to the Boris method as given in the second line vanishes for linear $A(x)$. In the situation of the magnetic field of \eqref{ode}, we can therefore replace $A$ by $A_1$ in \eqref{varint-borislike}. The variational integrator coincides with the Boris algorithm in the case of a constant magnetic field ($B_1\equiv0$).

This method is again complemented with the velocity approximation~\eqref{vn}.  
It can be given a one-step formulation similar to the Boris algorithm, with the correction term of \eqref{varint-borislike} added in the second line of \eqref{boris-1}.
It is, however, an implicit method, because the vector potential $A$ is evaluated at the new position $x^{n+1}$.

For the case of a strong magnetic field and for step sizes with $h|B|\le \mathit{Const.}$, the variational integrator has been shown to have excellent near-preservation of energy and magnetic moment over very long times \cite{hairer20lta}.

For large step sizes \eqref{eps-h2}, the variational integrator requires the same modification of the starting velocity as the Boris method in order to suppress high oscillations of large amplitude in the numerical solution.

\subsection{Filtered variational integrator}
As a new method to be studied here, we propose the following modification of the variational integrator:
with the filter functions 
\begin{align*}
&\psi(\zeta)=\tanch(\zeta/2)=\frac{\tanh(\zeta/2)}{\zeta/2},\qquad 
\varphi(\zeta)=\frac 1{\sinch(\zeta)}=\frac\zeta{\sinh(\zeta)},
\end{align*}
which are even functions and take the value $1$ at $\zeta=0$,
and with the skew-symmetric matrix $\widehat B_0$ defined by $-\widehat B_0 v =  v \times B_0$ for all $v\in\real^3$, we define the filter matrices
\begin{align*}
&\Psi=\psi\Bigl(-\frac h\eps\, \widehat B_0 \Bigr) = I + \biggl(1-\tanc\Bigl( \frac h{2\eps} \Bigr)\biggr) \widehat B_0^2, 
\\[1mm]
&\Phi=\varphi\Bigl(-\frac h\eps\, \widehat B_0 \Bigr) = I + \biggl(1-\sinc\Bigl( \frac h{\eps} \Bigr)^{-1}\biggr) \widehat B_0^2,
\end{align*}
where the rightmost expressions are obtained from a Rodriguez formula; see \cite[Appendix]{hairer20afb}. Here, $\tanc(\xi)=\tan(\xi)/\xi$ and $\sinc(\xi)=\sin(\xi)/\xi$.
The filter matrices $\Psi$ and $\Phi$ are symmetric and act as
the identity on vectors in the direction of $B_0$.

We put the filter matrix $\Psi$ in front of the right-hand side of \eqref{varint}:
\begin{align}\label{fvi}
 &\frac{x^{n+1}-2 x^n + x^{n-1}}{h^2} = 
 \\[2mm]
 \nonumber
 &\qquad \Psi \Bigl(A'(x^n)^\top \frac {(x^{n+1}-x^{n-1})}{2h}
-  \frac{A(x^{n+1})-A(x^{n-1})}{2h} +E^n \Bigr).
\end{align}
This is combined with the velocity approximation
\begin{equation}\label{vn-fvi}
v^n=\Phi \,\frac{x^{n+1}-x^{n-1}}{2h} + \eps \biggl(1-\sinc\Bigl( \frac h{\eps} \Bigr)^{-1}\biggr) E^n\times B_0.
\end{equation}
This filtered variational integrator coincides with the filtered Boris algorithm of \cite{hairer20afb} for the special case of a constant magnetic field $B(x)=B_0/\eps$. If additionally also $E$ is constant,
then this method yields the exact position and velocity, as was shown for the filtered Boris algorithm. 

For stepsizes $h$ with $\tan(h/(2\eps))\ge c >0$, the filter matrix $\Psi$ is positive definite. The above integrator can then be interpreted as a variational integrator corresponding to a discrete Lagrangian where the kinetic energy term has the modified mass matrix $\Psi^{-1}$. Its eigenvalues corresponding to the eigenvectors orthogonal to $B_0$ are $1/\tanc(h/(2\eps))$ and are thus proportional to $h/\eps$, which is greater than $h^{-1}$ under condition \eqref{eps-h2}. The discrete Lagrangian reads
\[
\begin{aligned}
L_h(x^n,x^{n+1})=&\ \frac{h}{2}(v^{n+1/2})^\top \Psi^{-1} v^{n+1/2}
\\
&\ +h\,\frac{A(x^n)^\top+A(x^{n+1})^\top}{2} \,v^{n+1/2}
-h\,\frac{\phi(x^n)+\phi(x^{n+1})}{2},
\end{aligned}
\]
where $v^{n+1/2} = (x^{n+1}-x^n)/h$. The standard (unfiltered) variational integrator has the same discrete Lagrangian except for the identity matrix in place of the matrix $\Psi^{-1}$.

The filtered variational integrator for \eqref{ode} can be written and implemented as the following implicit one-step  method:
$$
\begin{aligned}
v_{+}^{n-\tfrac{1}{2}}&=v^{n-\frac{1}{2}}+\frac{h}{2}\Psi E^n\\
v^{n+\tfrac{1}{2}}_{-}-v^{n-\frac{1}{2}}_{+}&=h\Psi
\biggl(\frac12 (v^{n+\frac{1}{2}}_{-}+v^{n-\tfrac{1}{2}}_{+})\times B^n 
\\
&\qquad\qquad 
+A'_1(x^n)\frac12 (v^{n+\tfrac{1}{2}}_{-}+v^{n-\tfrac{1}{2}}_{+})-\frac{A_1(x^{n+1})-A_1(x^{n-1})}{2h}\biggr)\\
v^{n+\tfrac{1}{2}}&=v^{n+\tfrac{1}{2}}_{-}+\frac{h}{2}\Psi E^n\\
x^{n+1}&=x^n+hv^{n+\tfrac{1}{2}}.
\end{aligned}
$$
This can be solved by a fixed-point iteration for $x^{n+1}$, where a good starting iterate is obtained from a Boris step.
The first velocity is chosen as follows:  we set $v^{1/2} = \bar v + \tfrac12 \delta v$ with $h\bar v =\tfrac12(x^1-x^{-1})$ and $h\,\delta v=x^1-2x^0+x^{-1}$ , where in view of \eqref{vn-fvi} for $n=0$,
$$
\bar v = \Phi^{-1} \Bigl( v^0 -  \eps \biggl(1-\sinc\Bigl( \frac h{\eps} \Bigr)^{-1}\biggr) E^n\times B_0 \Bigr)
$$
and $\delta v$ is implicitly determined  (and computed via fixed-point iteration) from \eqref{fvi} with $n=0$, i.e. from the equation
\begin{align*}
\delta v = h \Psi \Bigl( \bar v \times B(x^0) + A_1'(x^0)\bar v - \frac{A_1(x^{1})-A_1(x^{-1})}{2h} +E(x^0) \Bigr),
\end{align*}
where $x^{\pm1} = x^0 \pm h \bar v + \tfrac12 h \,\delta v$.

In contrast to the Boris algorithm and the unfiltered variational integrator, we here take the original initial data $v^0=\dot x(0)$ and $x^0=x(0)$.

\section{Modulated Fourier expansions}

We give modulated Fourier expansions of the
exact solution of \eqref{ode} and the numerical solutions of the three integrators for large step sizes $h^2\ge c\,\eps$ (in the following we set the irrelevant positive constant $c$ equal to 1 for simplicity). Analogous expansions for step sizes $h\le C\eps$ were previously given in \cite{hairer20lta,hairer20afb,hairer17smm}; see also \cite[Ch.\,XIII]{hairer06gni}. In particular, we explicitly state the differential equations for the dominant modulation functions up to $O(\eps^2)$  for the exact solution, and up to $O(h^2)$ for the numerical solutions.

\subsection{Modulated Fourier expansion of the exact motion}
We write the solution of \eqref{ode} as
\begin{equation}\label{exact_expansion}
x(t)\approx\sum_{k}z^k(t)e^{{\rm{i}}kt/\varepsilon}
\end{equation}
with coefficient functions $z^k(t)$ for which all time derivatives are bounded independently of $\varepsilon$. 

We diagonalize the linear map $v\mapsto v\times B_0$, which has eigenvalues $\lambda_1=\mathrm i $, $\lambda_0=0$ and $\lambda_{-1}=-\mathrm i $ (recall the normalization $|B_0|=1$). The normalized eigenvectors are denoted $v_1,v_0=B_0,v_{-1}=\overline v_1$. We let $P_j=v_jv_j^*$ be the orthogonal projections onto the eigenspaces. We write the coefficient functions of (\ref{exact_expansion}) in the basis $(v_j)$,
\[
z^k =z_1^k+z_0^k+z_{-1}^k, \quad z_j^k(t)=P_jz^k(t).
\]
The following theorem is a variant of Theorems 4.1 in \cite{hairer20lta,hairer20afb}, proved by the same arguments but in a technically simplified way, since here we have the constant frequency $1/\eps$ and constant projections $P_j$, as opposed to the state-dependent frequency and projections in \cite{hairer20lta,hairer20afb}.
\begin{theorem}\label{thm:mfe}
Let $x(t)$ be a solution of \eqref{ode} with an initial velocity bounded independently of $\eps$ $(|\dot x(0)|\le C_1)$, which stays in a compact set $K$ for $0\leq t\leq T$ (with $K$ and $T$ independent of $\eps$). For an arbitrary truncation index $N\geq 1$ we then have an expansion
\[
x(t)=\sum_{|k|\leq N}z^k(t)\mathrm e^{\mathrm i k t/\varepsilon}+R_N(t)
\]
with the following properties:
\begin{itemize}
\item [(a)] 
The modulation functions $z^k$ together with their derivatives (up to order $N$) are bounded as $z_j^0=O(1)$ for $j\in\{-1,0,1\}$, $z_1^1=O(\varepsilon)$, $z_{-1}^{-1}=O(\varepsilon)$, 
and for the remaining $(k,j)$ with $|k|\leq N$,
\[
z_j^k=O(\varepsilon^{|k|+1}).
\]
They are unique up to $O(\varepsilon^{N+2})$ and are chosen to satisfy $z^{-k}_{-j} = \overline{z^k_j}$.
Moreover, $\dot z^0_{\pm1}$ together with its derivatives is
  bounded as $\dot z^0_{\pm1} = \bigo (\eps )$.
\item [(b)] 
The remainder term and its derivative are bounded by
\[
R_N(t)=O(t^2\varepsilon^N),\quad \dot{R}_N(t)=O(t\varepsilon^N) \quad \text{for} \quad 0\leq t\leq T.
\]
\item [(c)]  
The functions $z_0^0$, $z_{\pm1}^0$, $z_1^1$, $z_{-1}^{-1}$ satisfy the differential equations
\[
\begin{aligned}
\ddot{z}^0_0&=P_0\bigl(\dot{z}^0\times B_1(z^0)+E(z^0)\bigr)+2P_0 \, {\rm Re}\Bigl(\frac{{\rm i}}{\varepsilon}\,z^1\times B_1'(z^0)z^{-1}\Bigr)+O(\varepsilon^2),\\
\dot{z}_{\pm1}^0&=\pm{\rm i}\varepsilon P_{\pm 1}\bigl(\dot{z}^0\times B_1(z^0)+E(z^0)\bigr)+O(\varepsilon^2),\\[1mm]
\dot{z}^{\pm 1}_{\pm 1}&=P_{\pm 1}\bigl(z^{\pm 1}_{\pm 1}\times B_1(z^0)\bigr)+O(\varepsilon^2).
\end{aligned}
\]
All other modulation functions $z_j^k$ are given by algebraic expressions depending on $z^0$, $\dot{z}_0^0$, $z_1^1$, $z_{-1}^{-1}$.
\item [(d)]  
Initial values for the differential equations of item (c) are given by
\[
\begin{aligned}
z^0(0)&=x(0)+\varepsilon \dot{x}(0)\times B_0+O(\varepsilon^2),\\
\dot{z}_0^0(0)&=P_0\dot{x}(0) -\eps P_0 \bigl((\dot x(0)\times B_0)\times B_1(x(0))\bigr) + O(\varepsilon^2),\\
z_{\pm 1}^{\pm 1}(0)&=\mp{\mathrm i}{\varepsilon} P_{\pm 1}\dot{x}(0)+O(\varepsilon^2).
\end{aligned}
\]
\end{itemize}
The constants symbolized by the $O$-notation are independent of $\varepsilon$ and $t$ with $0\leq t\leq T$, but depend on $N$, on the velocity bound $M$, on bounds of derivatives of $B_1$ and $E$ on the compact set $K$, and on $T$.
\end{theorem}

\subsection{Resonant modulated Fourier expansion of the Boris algorithm and the standard variational integrator for  $h^2\ge \eps$}
When the Boris method is applied to the linear differential equation $\ddot x = \dot x \times B_0/\eps$ with $|B_0|=1$ (that is, $B_1$ and $E$ are not present in \eqref{ode}), then diagonalization of $B_0$ shows that $x^n$ is a linear combination (with coefficients independent of $n$) of terms $1$, $nh$ and $\e^{\pm \iu nh\omega}$ , where
$$
h\omega  = 2\arctan\Bigl(\frac h{2\eps}\Bigr).
$$
If $h/\eps$ is large, then $h\omega $ is close to $\pi$. In particular, if $h^2 \ge \eps$, then $h\omega  = \pi - \gamma h$ with $\gamma>0$ bounded independently of $h$ and $\eps$ with $h^2\ge \eps$, and so
$\e^{\pm \iu nh\omega }= (-1)^n \e^{\mp \iu nh\gamma}$, where we note that $\e^{\mp \iu t\gamma}$ is a smooth function of $t$ all of whose derivatives are bounded independently of $\eps$ and $h$. In the general case of \eqref{ode}, we have the following result.

\begin{theorem}\label{thm:mfe-boris}
Let $x^n$ be the numerical solution obtained by applying either the Boris algorithm or the variational integrator to \eqref{ode} with a stepsize $h$ satisfying
\begin{equation}\label{hh-eps}
h^2 \ge \eps.
\end{equation}
We assume that the starting velocity $v^0$ is bounded independently of $\eps$ and $h$ and that its component orthogonal to $B_0$, 
i.e. $v^0_\perp=(I-P_0)v^0$, is small:
\begin{equation}\label{v-init-boris}
|v^0_\perp| \le c_1 \eps .
\end{equation}
We further assume that the numerical solution $x^n$ stays in a compact set $K$ for $0\leq nh\leq T$ (with $K$ and $T$ independent of $\eps$ and $h$). For an arbitrary truncation index $N\geq 2$, we then have a decomposition
\begin{equation} \label{mfe-boris}
x^n=y(t) + (-1)^n z(t) +R_N(t), \qquad t=nh,
\end{equation}
with the following properties:
\begin{itemize}
\item [(a)] 
The  functions $y(t)$ and $z(t)$ together with their derivatives (up to order $N$) are bounded as $y=O(1)$, $z=O(h^2)$. They are unique up to $O(\varepsilon^{N+2})$. 
Moreover, we have $\dot{y}\times B_0=O
(\varepsilon)$ and $z\cdot B_0 =O(h^4)$.
\item [(b)] 
The remainder term is
bounded by
\[
R_N(t)=O(t^2 h^N) 
\quad \text{for} \quad 0\leq t\leq T.
\]
\item [(c)]  
The functions $y_j=P_j y$ $(j=0,\pm1)$ and $z_{\pm 1}=P_{\pm1}z$ satisfy the differential equations
\[
\begin{aligned}
\ddot{y}_0&=P_0\bigl(\dot{y}\times B_1(y)+E(y)\bigr)+O(h^2),\\
\dot{y}_{\pm1}&=\pm{\rm i}\varepsilon P_{\pm 1}\bigl(\dot{y}\times B_1(y)+E(y)\bigr)+O(\eps h^2),\\[1mm]
\dot{z}_{\pm 1}&=\mp 4\iu \frac\eps{h^2} {z}_{\pm 1} + O(\eps h^2).
\end{aligned}
\]
The function $z_0=P_0z$ is given by an algebraic expression depending on $y$, $\dot{y}_0$ and $z_{\pm 1}$.
\item [(d)]  
Initial values for the differential equations of item (c) are given by
\[
\begin{aligned}
y(0)&= x^0+O(h^2),\\
\dot{y}_0(0)&=P_0 v^0 +O(h^2),
\\
z_{\pm 1}(0)&=\mp\frac{\iu h^2}{4\eps} \,P_{\pm 1}\Bigl( v^0 \mp \iu\eps \bigl(P_0v^0\times B_1(x^0)+E(x^0)\bigr)\Bigr)
+O(h^4).
\end{aligned}
\]
\end{itemize}
The constants symbolized by the $O$-notation are independent of $\varepsilon$, $h$ and $n$ with $0\leq nh \leq T$, but depend on the velocity bound, on bounds of derivatives of $B_1$ and $E$ on the compact set $K$, and on $T$.
\end{theorem}


We note that the differential equations for $y$ agree with those for $z^0$ of the exact solution up to $O(h^2)$. The differential equations for $z_{\pm1}$ and for $z^{\pm 1}_{\pm 1}$ of the exact solution differ, but we still have
$$
\frac{\d}{\d t}\, |z_{\pm 1}|^2 = 2\,\Re z_{\pm 1}^* \dot z_{\pm 1} = O\bigl( \eps |z_{\pm 1}|^2 \bigr) + O(\eps h^N) = O(\eps h^4),
$$
which is to be compared with
$$
\frac{\d}{\d t}\, |z^{\pm 1}_{\pm 1}|^2 = 2\,\Re (z^{\pm 1}_{\pm 1})^* \dot z^{\pm 1}_{\pm 1} = O(\eps^3).
$$
To obtain an $O(h^2)$ approximation to the guiding centre $z^0(t)$ over bounded time intervals, we run the Boris algorithm with the modified initial velocity $v^0= P_0\dot x(0)$ instead of $\dot x(0)$, or even better, determine $P_{\pm1}v^0$ such that $z_{\pm 1}(0)=O(h^4)$, which holds true with the proposed choice \eqref{v-init-mod}. 

\begin{proof} The bounds of parts (a) and (b) are proved as in previous proofs of modulated Fourier expansions; see e.g. \cite{hairer20lta} and \cite[Ch.\,XIII]{hairer06gni}. Here we just show (c) and (d), assuming that the bounds of (a) and (b) are already available. 

To derive the differential equations of (c), we insert \eqref{mfe-boris} into the two-step formulation of the numerical method, expand $y(t\pm h)$ and $z(t\pm h)$ into Taylor series at $t$, expand the nonlinear functions $B_1$ and $E$ at $y(t)$ and separate the terms without and with the factor $(-1)^n$. This gives us the equations
\begin{align*}
\ddot y + O(h^2) &= \dot y \times \frac{B_0}\eps + \dot y \times B_1(y) + E(y) + O(h^2)
\\
-\frac 4{h^2} z - \ddot z + O(h^2) &= -\dot z \times \frac{B_0}\eps + \dot z \times B_1(y) + \dot y \times B_1'(y) z + E'(y) z + O(h^2).
\end{align*}
In the equation for $z$ we note that also $\ddot z$ and the last three terms on the right-hand side are $O(h^2)$ as $z$ and its derivatives are $O(h^2)$, and the indicated $O(h^2)$ terms are then actually $O(h^4)$.

Taking the projection $P_0$ on both sides of the differential equation for $y$ yields the stated second-order differential equation for $y_0$ on noting that $P_0( \dot y \times {B_0})=0$.
 Moreover, since 
$P_{\pm1}(\dot y \times B_0) = \pm\iu \dot y_{\pm 1}$, we obtain
$$
\mp \frac\iu\eps \dot y_{\pm 1} = - \ddot y_{\pm 1} + P_{\pm1} \bigl(\dot{y}\times B_1(y)+E(y)\bigr)+O(h^2).
$$
Differentiating this equation and multiplying with $\iu\eps$ yields $\ddot y_{\pm 1}=O(\eps)$, which is $O(h^2)$ under condition \eqref{hh-eps}. So we obtain the stated first-order differential equation for $y_{\pm1}$.

Taking the projection $P_0$ in the above equation for $z$ yields
$-\frac 4{h^2} z_0 = O(h^2)$, and hence $z_0=O(h^4)$. Taking the projections $P_{\pm1}$ yields
$$
-\frac 4{h^2} z_{\pm1} = \mp \frac\iu\eps \dot z_{\pm1} + O(h^2),
$$
which can be rearranged into the stated differential equation for $z_{\pm1}$.

In view of \eqref{mfe-boris} for $n=0$ and $z(0)=O(h^2)$, we have $y(0)=x^0 + O(h^2)$.
Since we obtain by inserting  \eqref{mfe-boris} for $n=-1,1$
$$
v^0 =\frac{x^1-x^{-1}}{2h} = \dot y(0) - \dot z(0) + O(h^2) ,
$$ 
we obtain  the stated expression for $\dot y_0(0)$ on taking the projection $P_0$. 
Taking the projections $P_{\pm1}$ and using the differential equations for $y_{\pm1}$ and $z_{\pm1}$, we arrive at the stated expression for
$z_{\pm1}(0)$.
\qed
\end{proof}

\subsection{Non-resonant modulated Fourier expansion of the filtered variational integrator for $h^2\ge \eps$}

As the filtered integrator is exact for the linear equation $\ddot x = \dot x \times B_0/\eps$, it has the same high frequency $1/\eps$. When this integrator is applied to \eqref{ode}, it has a modulated Fourier equation that is very similar to that of the exact solution given in Theorem~\ref{thm:mfe}.

\begin{theorem}\label{thm:mfe-fvi}
Let $x^n$ be a solution of the filtered variational integrator applied to \eqref{ode} with a stepsize $h$ satisfying
\begin{equation}\label{hh-fvi}
h^2 \ge \eps 
\end{equation}
and, for some $N\ge 1$, the non-resonance conditions 
\begin{equation}
\label{nonres}
\begin{aligned}
&\displaystyle\Bigl|\sin\Bigl(\frac{kh}{2\eps}\Bigr)\Bigr| \ge c >0, \quad
\Bigl|\cos\Bigl(\frac{kh}{2\eps}\Bigr)\Bigr| \ge c >0  \qquad (k=1,\dots,N),
\\
&\displaystyle\Bigl|\tan\Bigl(\frac{kh}{2\eps}\Bigr)- \tan\Bigl(\frac{h}{2\eps}\Bigr)\Bigr| \ge c >0
 \qquad (k=2,\dots,N),
\end{aligned}
\end{equation}
where $c$ is a positive constant.
We assume that the initial velocity $v^0=\dot x(0)$ is bounded independently of $\eps$ and $h$, as in \eqref{init}. We further assume that the numerical solution $x^n$ stays in a compact set $K$ for $0\leq nh\leq T$ (with $K$ and $T$ independent of~$\eps$ and~$h$).
We then have an expansion, at $t=nh$,
\begin{equation}\label{mfe-fvi}
x^n=\sum_{|k|\leq N}z^k(t)\mathrm e^{\mathrm i k t/\varepsilon}+R_N(t)
\end{equation}
with the following properties:
\begin{itemize}
\item [(a)] The bounds of parts (a) of Theorem~\ref{thm:mfe} for the modulation functions are valid also in this case, except $z^k_0=O(h\eps^{|k|})$ for $|k|\ge 1$.
\item[(b)] The remainder at $t=nh$ is bounded, for arbitrary $M>1$, by
$$
P_0 R_N(t) = O(t^2 h^M) + O(t^2 \eps^N),             \quad P_{\pm 1} R_N(t) = O(t^2 \eps h^{M-1}) + O(t^2\eps^N).
$$
\item [(c)]  
The functions $z_0^0$, $z_{\pm1}^0$, $z_1^1$, $z_{-1}^{-1}$ satisfy the differential equations
\[
\begin{aligned}
\ddot{z}^0_0&=P_0\bigl(\dot{z}^0\times B_1(z^0)+E(z^0)\bigr)+O(h^2),\\
\dot{z}_{\pm1}^0&=\pm{\rm i}\varepsilon P_{\pm 1}\bigl(\dot{z}^0\times B_1(z^0)+E(z^0)\bigr)+O(\eps h),\\[1mm]
\dot{z}^{\pm 1}_{\pm 1}&=\frac{\eps}h\, \sin\Bigl(\frac h\eps \Bigr) P_{\pm1}\bigl(z^{\pm 1}_{\pm 1}\times B_1(z^0)\bigr) +O(\eps^2).
\end{aligned}
\]
All other modulation functions $z_j^k$ are given by algebraic expressions depending on $z^0$, $\dot{z}_0^0$, $z_1^1$, $z_{-1}^{-1}$.
\item [(d)]  
Initial values for the differential equations of item (c) are given by
\[
\begin{aligned}
z^0(0)&=x^0+O(h^2), \\ 
\dot{z}_0^0(0)&=P_0v^0 + O(h^2),\\
z_{\pm 1}^{\pm 1}(0)&=\mp{\mathrm i}{\varepsilon} P_{\pm 1}v^0+O(\varepsilon h).
\end{aligned}
\]
\end{itemize}
The constants symbolized by the $O$-notation are independent of $\varepsilon$ and $t$ with $0\leq t\leq T$, but depend on $m$ and $N$, on the velocity bound~\eqref{init}, on bounds of derivatives of $B_1$ and $E$ on the compact set $K$, and on $T$.
\end{theorem}


\begin{proof} Parts (a) and (b) are again proved as in previous proofs of modulated Fourier expansions; see e.g.~\cite{hairer20lta} and \cite[Ch.\,XIII]{hairer06gni}. Here we only show (c) and (d), assuming that the bounds of (a) and (b) are already available. 

To derive the differential equations of (c), we insert \eqref{mfe-fvi} into the two-step formulation of the numerical method, expand $z^k(t\pm h)$  into a Taylor series at $t$, use Lemma~5.1 of \cite{hairer20afb} to expand the first and second-order difference quotients for $z^k(t)\e^{\iu kt/\eps}$ for $0<|k|\le N$, and expand $B_1$ and $E$ at $z^0(t)$. We then separate the terms multiplying $\e^{\iu k t/\eps}$ for $|k|\le N$.
Moreover, we consider the components $z^k_j=P_jz^k$ for $j=0,\pm1$. 

For $k=0$, $j=0$ we obtain
$$
\ddot z^0_0 + O(h^2) = P_0 \Bigl((\dot z^0+O(h^2)) \times B_1(z^0)+ E(z^0) + O(\eps^2/h)\Bigr),
$$
where the $O(h^2)$ terms result from the Taylor expansions of the second and first order difference quotients of $z^0$, and the (smaller) $O(\eps^2/h)$ term results from the Taylor expansion of $B_1$ and $E$ at $z^0$ and the bound $z^k=O(\eps^{|k|})$. This yields the first equation of (c).

For $k=0$, $j=1$ we obtain
\begin{align*}
\ddot z^0_1 + O(\eps h^2)  = \frac {2\eps}h \, \tan\Bigl(\frac h{2\eps} \Bigr) &
\biggl( \frac\iu\eps (\dot z^0_1+O(h^2)) \biggr.
\\
\biggl. +&\ P_1 \Bigl((\dot z^0+O(h^2)) \times B_1(z^0)+ E(z^0) + O(\eps^2/h)\Bigr) \biggr).
\end{align*}
We solve this equation for $\dot z^0_1$, which appears in the dominant term with a factor $h^{-1}$, and recall that $|\tan(h/(2\eps))|\ge c >0$
by the non-resonance condition~\eqref{nonres}. Using that $\ddot z^0_1$ and its higher derivatives are $O(\eps)$ by part (a), this yields
$$
\dot z^0_1 = \iu\eps P_1 \Bigl(\dot z^0 \times B_1(z^0)+ E(z^0) \Bigr) +  O(\eps h),
$$
which is the differential equation for $z^0_1$ stated in (c). The case $j=-1$ is obtained by taking complex conjugates.

For $k=1$, $j=1$ we find for $y^1_1(t)=  z^1_1(t)\e^{\iu t/\eps}$, using Lemma~5.1 of \cite{hairer20afb} and the $O(\eps)$ bound for $z^1_1$ and its derivatives of part (a),
\begin{align*}
&\frac{y^1_1(t+h)-2y^1_1(t)+y^1_1(t-h)}{h^2} 
\\
&\quad = \e^{\iu t/\eps} \biggl( -\frac 4{h^2} \, \sin^2\Bigl(\frac h{2\eps} \Bigr) z^1_1(t) +
\frac{2\,\iu}h \,  \sin\Bigl(\frac h{\eps} \Bigr) \dot z^1_1(t) + O(\eps) \biggr)
\end{align*}
and
\begin{align*}
& \frac{y^1_1(t+h)-y^1_1(t-h)}{2h}  
= \e^{\iu t/\eps} \,\biggl(\frac\iu h  \sin\Bigl(\frac h{\eps} \Bigr) z^1_1(t) + \cos\Bigl(\frac h{\eps} \Bigr) \dot z^1_1(t)+ O(\eps h) \biggr),
\end{align*}
and hence
\begin{align*}
&P_1 \, \mathrm{tanch}\Bigl( \frac h{2\eps} \widehat B_0 \Bigr) \biggl( \frac{y^1_1(t+h)-y^1_1(t-h)}{2h} \times \frac{B_0}\eps \biggr)
\\
&= \    \frac {2\eps}h \, \tan\Bigl(\frac h{2\eps} \Bigr) \frac\iu\eps \,
\e^{\iu t/\eps} \,\Bigl(\frac\iu h  \sin\Bigl(\frac h{\eps} \Bigr) z^1_1(t) + \cos\Bigl(\frac h{\eps} \Bigr) \dot z^1_1(t)+ O(\eps h) \Bigr)
\\
&= \ \e^{\iu t/\eps} \biggl( -\frac 4{h^2} \, \sin^2\Bigl(\frac h{2\eps} \Bigr) z^1_1(t)  
+ \frac{2\,\iu}h \, \tan\Bigl(\frac h{2\eps} \Bigr)\cos\Bigl(\frac h{\eps} \Bigr) \dot z^1_1(t) +
O(\eps) \biggr).
\end{align*}
Inserting \eqref{mfe-fvi} into the two-step formulation of the filtered variational integrator and collecting the terms with factor $\e^{\iu t/\eps} $, we thus obtain
\begin{align*}
& -\frac 4{h^2} \, \sin^2\Bigl(\frac h{2\eps} \Bigr) z^1_1(t) +
\frac{2\,\iu}h \,  \sin\Bigl(\frac h{\eps} \Bigr) \dot z^1_1(t) + O(\eps)
\\
&= \ -\frac 4{h^2} \, \sin^2\Bigl(\frac h{2\eps} \Bigr) z^1_1(t)  
+ \frac{2\,\iu}h \, \tan\Bigl(\frac h{2\eps} \Bigr)\cos\Bigl(\frac h{\eps} \Bigr) \dot z^1_1(t) +
O(\eps) 
\\
&\quad +\    \frac {2\eps}h \, \tan\Bigl(\frac h{2\eps} \Bigr) \, P_1 \Bigl( \frac\iu h  \sin\Bigl(\frac h{\eps} \Bigr) z^1_1(t)\times B_1(z^0(t)) 
+ O(\eps) \Bigr) .
\end{align*}
Here the dominant terms are the first terms on the left-hand and the right-hand sides, which are the same and thus cancel. The dominant terms then become the terms containing the factor $(2\,\iu/h)\dot z^1_1(t)$. Since a calculation shows that we have, with $\xi=h/(2\eps)$ for short,
$$
\sin(2\xi) - \tan(\xi)\cos(2\xi) = (\tan(2\xi) - \tan(\xi))\cos(2\xi)= \tan(\xi),
$$
the above equation yields the differential equation for $z^1_1$ as stated in part (c) of the theorem. The result for $z^{-1}_{-1}$ is obtained by taking complex conjugates.
%

The formulae for the initial values are obtained by the same arguments as in the proof of Theorem~\ref{thm:mfe-boris}, using here that
$(x^1-x^{-1})/(2h)$ is related to $v^0$ by \eqref{vn-fvi} for $n=0$.
\qed
\end{proof}

\section{Time scale $\eps^0$: 
error bounds for position and parallel velocity}

Comparing the modulated Fourier expansions of the numerical solution with that of the exact solution, we obtain the following error bounds from Theorems~\ref{thm:mfe}--\ref{thm:mfe-fvi}.

\begin{theorem} Consider applying the Boris method, the variational integrator and the filtered variational integrator to \eqref{ode} over a time interval $0\le t \le T$ (with $T$ independent of $\eps$) using a stepsize $h$ with 
$$
h^2\ge \eps.
$$
Suppose that the conditions of Theorem~\ref{thm:mfe-boris} are satisfied in the case of the Boris method and the variational integrator (in particular, small perpendicular starting velocity: $v_\perp^0=O(\eps)$), and that the conditions of 
Theorem~\ref{thm:mfe-fvi} are satisfied in the case of the filtered variational  integrator (in particular, the non-resonance conditions \eqref{nonres} and bounded initial velocity \eqref{init}). For each of the three methods, the errors in position $x$ and parallel velocity $v_\parallel=P_0v$ at time $t_n=nh \le T$  are then bounded by
$$
|x^n- x(t_n)| \le C h^2, \qquad |v^n_\parallel - v_\parallel(t_n)| \le C h^2 \qquad\quad (t_n\le T),
$$
where $C$ is independent of $\eps$, $h$ and $n$ with $h^2\ge \eps$ and $nh\le T$ (but depends on $T$).
\end{theorem}

\begin{proof} The result is obtained by representing the exact and numerical solutions by their modulated Fourier expansions and using the bounds and differential equations of the modulation functions as given in Theorems~\ref{thm:mfe}--\ref{thm:mfe-fvi}. Note that the differential equations of the dominating modulation functions for the three methods and for the exact solution coincide up to defects of size $O(h^2)$, which lead to an $O(h^2)$ error in the positions. Inserting the modulated Fourier expansion of the numerical solution into the formula for the approximate velocity $v^n$ for each method and comparing with the time-differentiated modulated Fourier expansion of the exact solution then yields the $O(h^2)$ error bound for the parallel velocity.
\qed
\end{proof}

\begin{remark}
For $h^2\sim\eps$, the above error bounds are thus $O(\eps)$. For all three methods, the error bounds remain in general $O(\eps)$ also for smaller stepsizes $h\sim\eps$. This can be shown by comparing the modulated Fourier expansions for such stepsizes, as given in
\cite{hairer20lta} for the standard variational integrator.
The filtered Boris method of \cite{hairer20afb}, used with $h\sim\eps$,  has an $O(\eps^2)$ error in the position and the parallel velocity, and an $O(\eps)$ error in the perpendicular velocity.
\end{remark}

\begin{figure}[!ht]
\centerline{\includegraphics[scale=0.5]{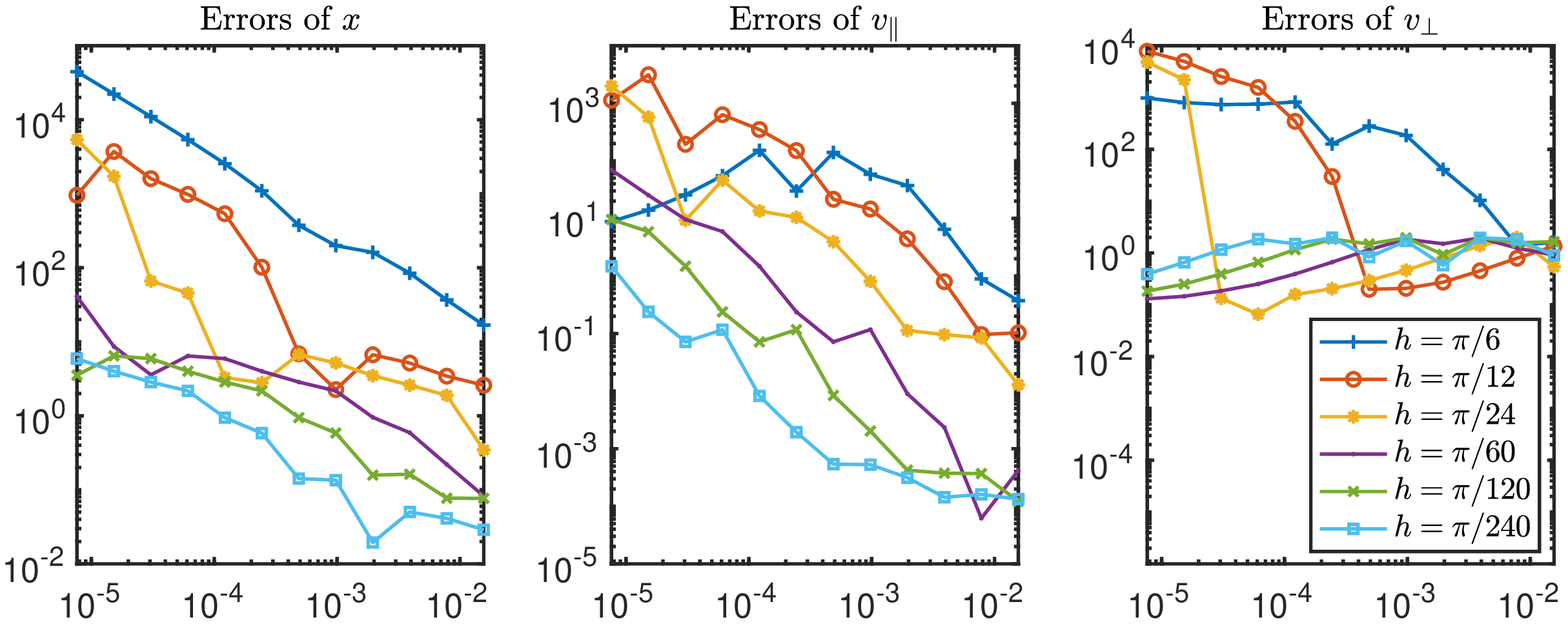}}
\centerline{\includegraphics[scale=0.5]{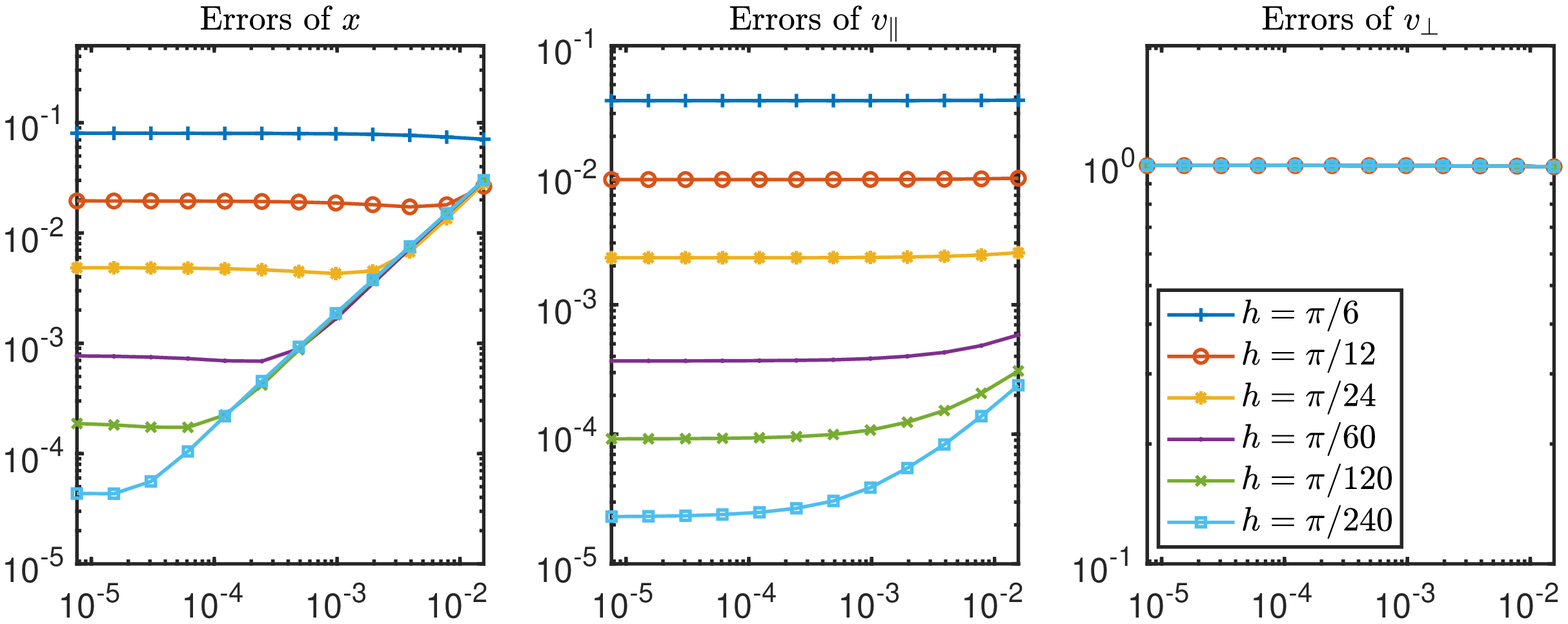}}
\centerline{\includegraphics[scale=0.5]{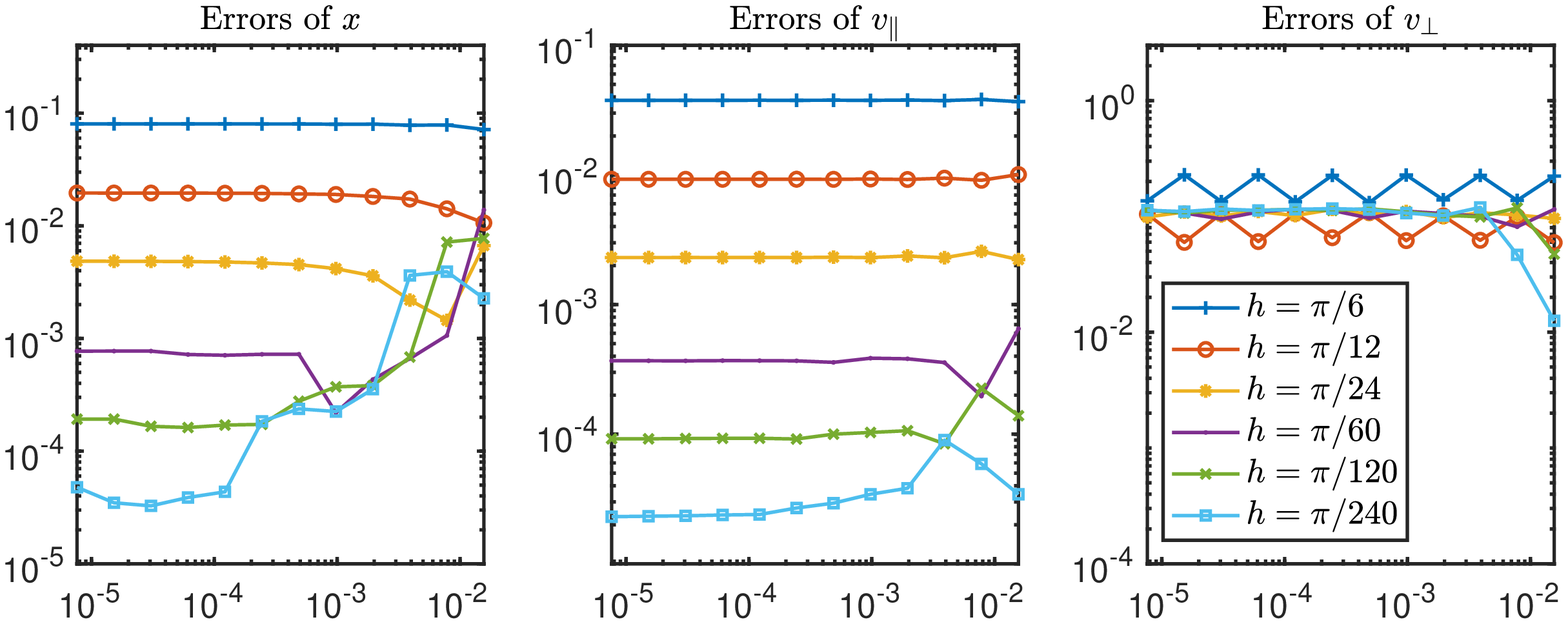}}
\caption{Global error vs.~$\eps$ ($\eps=1/2^j, j=6,\cdots 17$) with different $h$ for the Boris algorithm with starting values $x(0)$, $v(0)$ (top row), with modified starting values \eqref{v-init-mod} (centre row), and for the filtered variational integrator with starting values $x(0)$, $v(0)$ (bottom row).}\label{fig:order}
\end{figure}
\medskip\noindent
{\bf Numerical experiment.}
For the example of Section 2, Figure 5.1 
shows the relative errors in $x$, $v_\parallel$ and $v_\perp$ at time $t=\pi/2$ versus $\eps$  for various step sizes $h$ for three numerical approaches: 
\begin{itemize}
\item[(i)] in the top row for the Boris algorithm with the original initial data as starting values, 
\item[(ii)] in the centre row for the Boris algorithm with modified starting values \eqref{v-init-mod},
\item[(iii)] in the bottom row for the filtered variational integrator with the original initial data as starting values.
\end{itemize}
For any step size $h$, the errors in $x$ and $v_\parallel$ increase  roughly proportionally to $h^2/\eps$ when $\eps\rightarrow 0$ in case (i), whereas in cases (ii) and (iii) the errors tend to a constant error level proportional to $h^2$.

\section{Time scale $\eps^{-1}$: perpendicular drift}
\subsection{Perpendicular drift of the exact motion}
We let  $P_\parallel=P_0=B_0B_0^\top$ be the orthogonal projection onto the span of $B_0$, and
$P_\perp=P_1+P_{-1}=I-P_\parallel$ the orthogonal projection onto the plane orthogonal to $B_0$. We decompose $x\in\real^3$ as
$$
x=x_\parallel + x_\perp \quad\text{ with }\quad x_\parallel = P_\parallel x, \ \ x_\perp=P_\perp x.
$$
We assume that (with slight abuse of notation for $B_1$)
\begin{equation}\label{BE}
B_1(x) = B_1(x_\perp)+\eps B_2(x), \quad E(x) = E_\perp(x_\perp) + E_\parallel(x)+ \eps E_2(x),
\end{equation}
with $E_\perp\cdot B_0=0$ and $E_\parallel\times B_0=0$, and
where the functions $B_1,B_2$ and $E_\perp,E_\parallel,E_2$ on the right-hand side and all their derivatives are bounded independently of $\eps$. We thus only allow a weak dependence of the magnetic field and the perpendicular electric field on $x_\parallel$. We then have the following result.
%

\begin{theorem} \label{thm:drift}
Let $x(t)$ be a solution of \eqref{ode} with \eqref{BE}, with an initial velocity bounded independently of $\eps$ $(|\dot x(0)|\le M)$, which stays in a compact set $K$  for $0\le t \le c\,\eps^{-1}$
(with $K$  and $c$ independent of $\eps$). Then,  the solution $y_\perp(t)$ of the initial-value problem for the slow differential equation
\begin{equation}\label{ode-perp}
\dot y_\perp(t) = \eps E_\perp(y_\perp(t)) \times B_0, \qquad y_\perp(0)=x_\perp(0),
\end{equation}
remains $O(\eps)$-close to the perpendicular component of $x(t)$ 
over times $O(\eps^{-1})$:
\begin{equation} \label{xperp-approx}
| x_\perp(t) - y_\perp(t) | \le C\eps, \qquad 0\le t \le c/\eps.
\end{equation}
The constant $C$ is independent of $\varepsilon$ and $t$ with $0\leq t\leq c/\eps$, but depends on the initial velocity bound $M$, on bounds of derivatives of $B_1$ and $E$ on the compact set $K$, and on $c$.
\end{theorem}

\begin{remark} It is well known in the physical literature (going back to \cite[Eq.~(13)]{northrop63tam}) that the perpendicular velocity is largely determined by the $E\times B$ term, as is justified by averaging techniques; see also, e.g., \cite[Eq.~(6)]{filbet16asp} in the numerical literature. An $O(\eps)$ bound over times $O(\eps^{-1})$  as in \eqref{xperp-approx} was recently proved in \cite{filbet20cao} in the more restricted setting of a constant magnetic field $(B_1\equiv 0)$ and an electric field with $E_\parallel\equiv 0$.
\end{remark}

\begin{proof} The proof uses the modulated Fourier expansion of Theorem~\ref{thm:mfe}, in particular the differential equations for $z^0_{\pm 1}$ and $z^{\pm 1}_{\pm 1}$ in part (c), and the familiar argument of Lady Windermere's fan \cite{hairer93sod}. We structure the proof into four parts (a)--(d).

(a) Over the (short) time interval $0\le t \le 1$, Theorem 4.1 yields that
$$
x_\perp(t) = z^0_\perp(t) + z^1_1(t)\e^{\iu t/\eps}+z^{-1}_{-1}(t)\e^{-\iu t/\eps}
 + O(\eps^2),  
$$
where $z^0_\perp(t)=z^0_1(t)+z^0_{-1}(t)$ and $z^{\pm 1}_{\pm 1}(t)$
satisfy the differential equations
\begin{align*}
\dot z^0_\perp &= \eps \Bigl((\dot z^0_\parallel + \dot z^0_\perp)\times B_1(z^0_\perp) + E_\perp(z^0_\perp) \Bigr)\times B_0 + O(\eps^2),
\\
\dot{z}^{1}_{1}&=P_{1}\bigl(z^{1}_{1}\times B_1(z^0_\perp)\bigr)+O(\varepsilon^2),
\end{align*}
and $z^{-1}_{-1} = \overline {z^1_1}$.
We note that 
$\dot z^0_\parallel=\dot z^0_0=\dot x_\parallel + O(\eps)$, because we have $\frac{\d}{\d t}\bigl(z^1_0 \e^{it/\eps}\bigr) =
(\iu z^1_0/\eps + \dot z^1_0 )\e^{it/\eps} = O(\eps)$.

(b)
On every time interval $n\le t \le n+1$ (with $n\le c/\eps$) we can do the same and, denoting by $y^{[n]}_\perp$ the function $z^0_\perp$ on this interval and by $z^{[n]}_1$ the function $z^1_1$, we have
$$
x_\perp(t) = y^{[n]}_\perp(t) + 2\,\Re \bigl(z^{[n]}_1\e^{\iu t/\eps} \bigr) + O(\eps^2),\quad\ n \le t \le n+1,
$$
where $y^{[n]}_\perp$ and $z^{[n]}_1$ solve the initial value problems
\begin{align*}
&\dot y^{[n]}_\perp = \eps \Bigl(\dot x_\parallel \times B_1(y^{[n]}_\perp) + E_\perp(y^{[n]}_\perp) \Bigr)\times B_0, 
\\
&y^{[n]}_\perp(n)= x_\perp(n) - 2\,\Re \bigl(z^{[n]}_1(n)\e^{\iu n/\eps} \bigr),
\end{align*}
and 
\begin{align*}
& \dot z^{[n]}_1=P_{1}\bigl(z^{[n]}_1\times B_1(y^{[n]}_\perp)\bigr), 
\\
& z^{[n]}_1(n) = \iu \eps P_1 \dot x(n).
\end{align*}
We consider these initial value problems on the time interval $n \le t \le c/\eps$. By Theorem~\ref{thm:mfe}, we have
\begin{align*}
y^{[n]}_\perp(n+1) &= y^{[n+1]}_\perp(n+1) + O(\eps^2),
\\
z^{[n]}_1(n+1) &=  z^{[n+1]}_1(n+1) + O(\eps^2).
\end{align*}
In view of the factor $\eps$ in front of the right-hand side of the differential equations for  $y^{[n+1]}_\perp$ and $y^{[n]}_\perp$,
this estimate implies that
$$
y^{[n+1]}_\perp (t) - y^{[n]}_\perp(t) = O(\eps^2), \qquad n+1 \le t \le c/\eps.
$$
Moreover, 
taking the inner product of the differential equation for $z^{[n]}_1$ with $z^{[n]}_1$ shows that 
$$
\frac{\d}{\d t} |z^{[n]}_1|^2=2\, \Re \overline{z^{[n]}_1}^\top \dot z^{[n]}_1 =0,
$$
and hence
$$
|z^{[n]}_1(t)| = |z^{[n]}_1(n)|, \qquad n \le t \le c/\eps.
$$
(c) Next we study the difference between $y^{[0]}_\perp(t)$ and  $y_\perp(t)$ of \eqref{ode-perp}. We have
\begin{align*}
 y^{[0]}_\perp(t)- y_\perp(t) =& \ \Bigl( y^{[0]}_\perp(0)- y_\perp(0)\Bigr) + \eps \int_0^t \Bigl(E_\perp(y^{[0]}_\perp(s)) - E_\perp(y_\perp(s))\Bigr)\times B_0 \,\d s 
\\
& +\ \eps \int_0^t \bigl(\dot x_\parallel(s) \times B_1(y^{[0]}_\perp(s)) \bigr) \times B_0 \,\d s.
\end{align*}
The difference of the initial values is $O(\eps^2)$, and the last integral term is bounded using partial integration:
\begin{align*}
& \eps \int_0^t \bigl(\dot x_\parallel(s) \times B_1(y^{[0]}_\perp(s)) \bigr) \times B_0\,\d s 
\\
&=
\eps \Bigl( x_\parallel(t) \times B_1(y^{[0]}_\perp)(t) - x_\parallel(0) \times B_1(y^{[0]}_\perp(0))\Bigr) \times B_0
\\
&\quad -\  \eps \int_0^t \Bigl( x_\parallel(s) \times \frac{\partial B_1}{\partial x_\perp}(y^{[0]}_\perp(s)) \,\dot y^{[0]}_\perp(s) \Bigr) \times B_0 \,\d s .
\end{align*}
This is $O(\eps)$ for $0\le t \le c/\eps$, because $x_\parallel$ is bounded by assumption and $\dot y^{[0]}_\perp(s) = O(\eps)$.
With a Lipschitz bound of $E$ and the Gronwall lemma, this yields that
the difference between $y^{[0]}_\perp(t)$ and  $y_\perp(t)$ of \eqref{ode-perp} is bounded by
$$
y^{[0]}_\perp(t)- y_\perp(t) = O(\eps), \qquad 0\le t \le c/\eps.
$$

(d) With the above estimates we obtain, for $n\le t \le n+1 \le c/\eps$,
\begin{align*}
x_\perp(t) - y_\perp(t) =& \ \Bigl(x_\perp(t)- y^{[n]}_\perp(t) -  2\,\Re \bigl(z^{[n]}_1(t)\e^{\iu t/\eps} \bigr) \Bigr) 
+ 2\,\Re \bigl(z^{[n]}_1(t)\e^{\iu t/\eps} \bigr) 
\\ 
&+\ \sum_{j=0}^{n-1} \Bigl( y^{[j+1]}_\perp(t) - y^{[j]}_\perp(t) \Bigr) +
\Bigl( y^{[0]}_\perp(t)- y_\perp(t) \Bigr)
\\
=& \ O(\eps^2) + O(\eps) + O(n\eps^2) + O(\eps) = O(\eps),
\end{align*}
which is the stated result.
\qed
\end{proof}

\subsection{Perpendicular drift of numerical approximations}

For the Boris algorithm with large step size \eqref{eps-h2} and a small perpendicular component of the starting velocity we obtain the following result from Theorem~\ref{thm:mfe-boris}.

\begin{theorem} \label{them:drift-boris} Under the assumptions of Theorem~\ref{thm:mfe-boris} (in particular \eqref{hh-eps}--\eqref{v-init-boris}), and provided that the numerical solution $x^n$ of the Boris method
stays in a compact set $K$  for $0\le t \le c\,\eps^{-1}$
(with $K$  and $c$ independent of $\eps$ and $h$), 
the solution $y_\perp(t)$ of the initial-value problem for the slow differential equation $\eqref{ode-perp}$
remains $O(h^2)$-close to the perpendicular component of $x^n$ 
over times $O(\eps^{-1})$:
\begin{equation} \label{xperp-approx-boris}
| x^n_\perp - y_\perp(t_n) | \le C h^2, \qquad\ 0\le t_n=nh \le c/\eps.
\end{equation}
The constant $C$ is independent of $\varepsilon$ and $h$ and $n$ with $0\leq nh \leq c/\eps$, but depends on the initial velocity bound, on bounds of derivatives of $B_1$ and $E$ on the compact set $K$, and on $c$.
\end{theorem}

\begin{proof} The proof uses Theorem~\ref{thm:mfe-boris} 
and Lady Windermere's fan in the same way as in the proof of Theorem~\ref{thm:drift}, without any additional difficulty. We therefore omit the details.
\qed
\end{proof}

\begin{figure}[!ht]
\centerline{\includegraphics[scale=0.5]{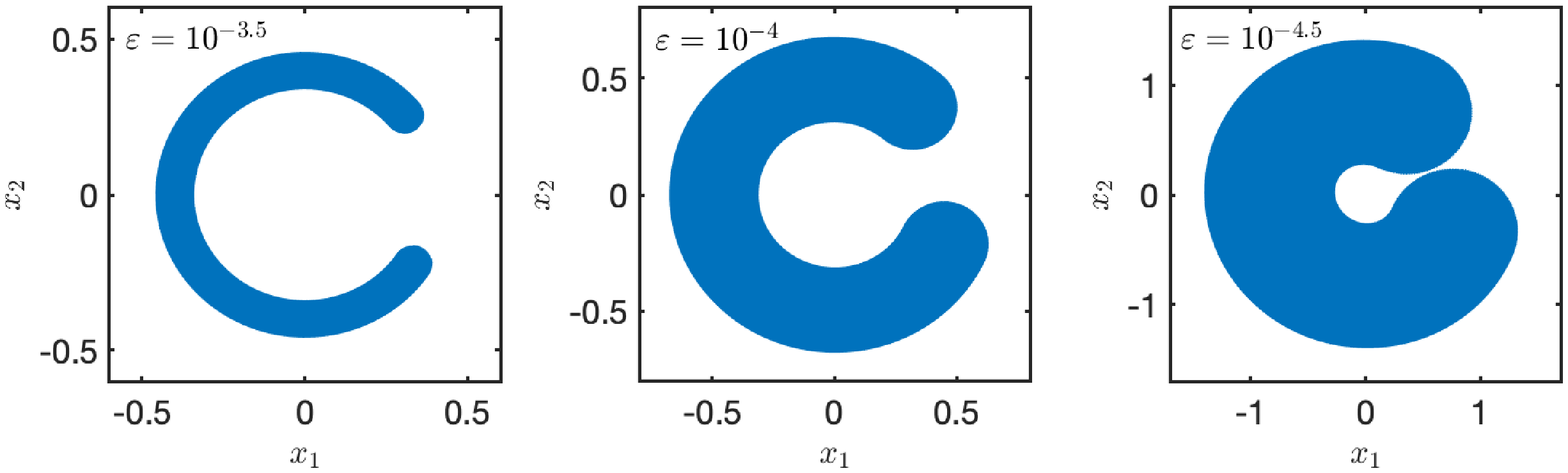}}
\centerline{\includegraphics[scale=0.5]{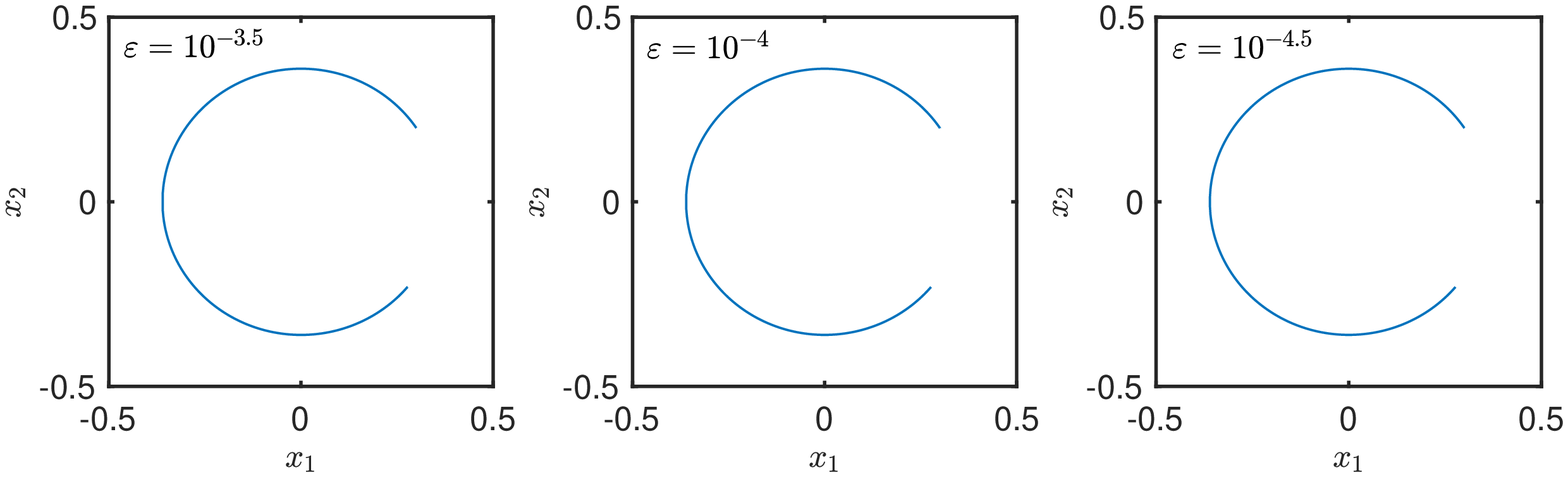}}
\centerline{\includegraphics[scale=0.5]{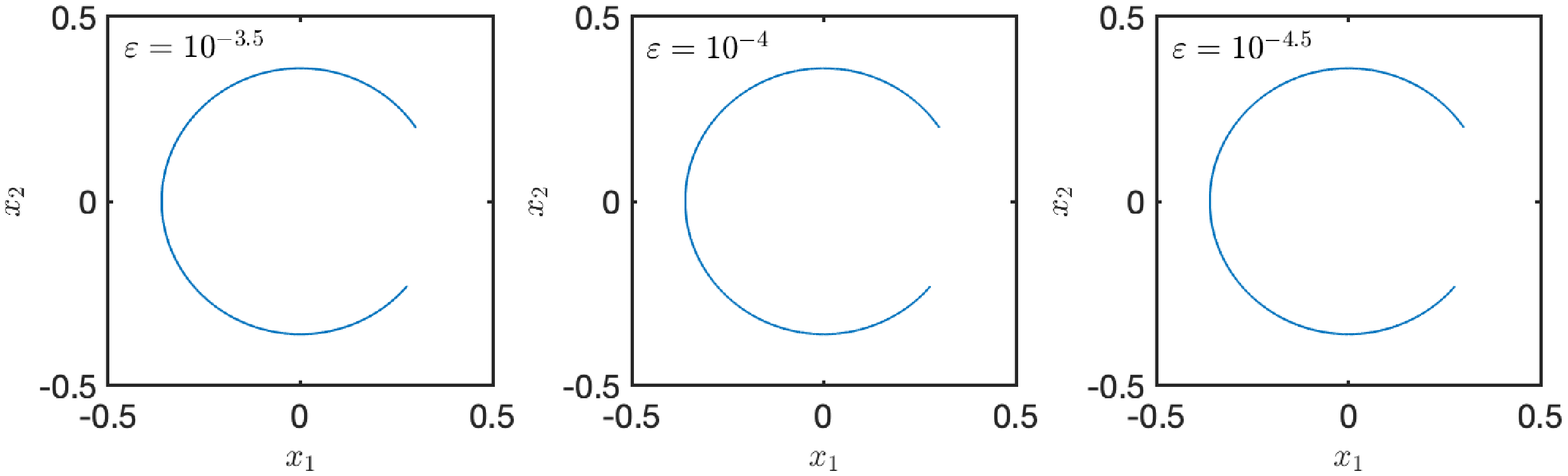}}
\caption{Particle trajectory for times $t\le 5/\eps$ projected onto the perpendicular plane as computed by the Boris algorithm with starting values $x(0)$, $v(0)$ (top row), with modified initial values \eqref{v-init-mod} (centre row), and by the filtered variational integrator with starting values $x(0)$, $v(0)$ (bottom row). The step size used is $h=10^{-2}$ in all cases.}\label{fig:motion}
\end{figure}

Analogous results hold true also for the standard and filtered variational integrators, for the latter with non-resonant stepsizes \eqref{nonres}, using the corresponding modulated Fourier expansions as given in Theorems~\ref{thm:mfe-boris} and~\ref{thm:mfe-fvi}. We note that for the filtered variational integrator we do not need the smallness assumption \eqref{v-init-boris} for the perpendicular component of the velocity required for the Boris and standard variational integrators, but the mere boundedness of the initial velocity suffices for the filtered variational integrator. However, in view of the $O(\eps h)$ remainder term (instead of $O(\eps h^2)$) in the differential equation for $z^0_{\pm1}$ in part (c) of Theorem~\ref{thm:mfe-fvi}, the error bound of $x^n_\perp$ for the filtered variational integrator is only $O(h)$ instead of $O(h^2)$.

\medskip\noindent
{\bf Numerical experiment.}
For the example of Section 2 and for the methods (i)--(iii) of the numerical experiments of Section 5, Figure 6.1 shows the projection of the computed particle trajectory onto the plane perpendicular to $B_0=e_3$ up to time $T=5/\eps$, for the fixed step size $h=10^{-2}$ and three values of $\eps$. It is observed that the Boris algorithm with the original initial velocity as starting velocity shows an enlarged gyroradius for $h\gg\eps$, while after modifying  the starting velocity to \eqref{v-init-mod}, the Boris algorithm shows correct results.
The same behaviour is observed also for the standard variational integrator (not shown here, since the pictures are indistinguishable). In contrast, the filtered variational integrator shows correct results both for the original initial values (as shown) and for the modified starting velocity (not shown here).

\section{Long-term near-conservation of magnetic moment and energy}

\subsection{Time scale $\eps^{-1}$: Standard variational integrator}

For the standard (unfiltered) variational integrator with step sizes \eqref{eps-h2} and the modified starting velocity \eqref{v-init-mod} we can show energy conservation up to $O(h^2)$ over time $\eps^{-1}$, provided that $h^6\le \eps$. We do not have, and do not expect, such a result for the Boris algorithm in a non-uniform magnetic field~\eqref{ode}.

\begin{theorem} \label{them:drift-boris} Under the assumptions of Theorem~\ref{thm:mfe-boris} , and provided that the numerical solution $x^n$ of the variational integrator with step size \eqref{eps-h2} and starting velocity \eqref{v-init-mod}
stays in a compact set $K$  for $0\le t \le c\,\eps^{-1}$
(with $K$  and $c$ independent of $\eps$ and $h$), 
the total energy \eqref{H} 
remains $O(h^2)$-close to the initial energy
over times $c\min(\eps^{-1},h^{-6})$:
\begin{equation} \label{H-vi}
| H(x^n,v^n) - H(x^0,v^0) | \le C h^2, \qquad\ 0\le nh \le c\min(\eps^{-1},h^{-6}).
\end{equation}
Moreover, with the modified initial velocity, the magnetic moment \eqref{I} remains $O(\eps^2)$ small over times $c\,\eps^{-1}$:
\begin{equation} \label{I-vi}
| I(x^n,v^n)| \le C \eps^2, \qquad\ 0\le nh \le c\,\eps^{-1}.
\end{equation}
The constants $C$ are independent of $\varepsilon$ and $h$ and $n$ with $0\leq nh \leq c/\eps$, but depend on bounds of derivatives of $B_1$ and $E$ on the compact set $K$, and on~$c$.
\end{theorem}

\begin{proof} The proof uses Theorem~\ref{thm:mfe-boris} and
arguments from the proof of Proposition 6.2 in \cite{hairer17smm}.
We first consider the energy behaviour over a short time interval of length $1$, over which we can apply Theorem~\ref{thm:mfe-boris}.
With $D=\d/\d t$ and the shift operator $\e^{hD}$,  with $\delta(\zeta)=(\zeta-\zeta^{-1})/2$ and $\rho(\zeta)=\zeta-2+\zeta^{-1}$, and with
the expansions $\delta(e^h)/(2h)=(1+\alpha_2 h^2 + \alpha_{4}h^{4} +\dots)$ and
$\rho(e^h)/h^2=(1+\beta_2 h^2 + \beta_{4}h^{4} +\dots)$,
we write the equation for the function $y(t)$ in the decomposition \eqref{mfe-boris} as
\begin{align} \nonumber
&\ddot y + \beta_2 h^2 y^{(4)} + \beta_4 h^4 y^{(6)} + \ldots = (\dot y + \alpha_2h^2 y^{(3)} + \alpha_4 h^4 y^{(5)}+\ldots )\times \frac{B_0}\eps
\\ \label{y-eq}
&\quad + A_1'(y)^\top \frac{\delta(\e^{hD})}{2h} y - \frac{\delta(\e^{hD})}{2h} A_1(y) - \nabla \phi(y) + O(|z|^2) + O(h^N),
\end{align}
where the left-hand side contains only {\it even\/}-order
derivatives of $y$, and the right-hand side contains only {\it odd\/}-order derivatives of $y$.
We multiply both sides of (\ref{y-eq}) with $\dot y^\top$. The multiplied left-hand side is
the time derivative of an expression 
in which the appearing
second and higher derivatives of $y$ can be substituted as functions of
$(y,\dot y)$ via the differential equation for $y$ in part (c) of Theorem~\ref{thm:mfe-boris}; cf.~\cite{hairer18ebo}.
On the right-hand side
we have
\begin{align} \label{right-hand}
&-\dot y^\top\eps^{-1} \widehat B_0 (\alpha_2h^2 y^{(3)} + \alpha_4 h^4 y^{(5)}+\ldots )
\\ \nonumber
&
+ \dot y^\top \frac1{2h}\Bigl(A_1'(y)^\top  \delta(e^{hD}) y -  \delta(e^{hD}) A_1(y)\Bigr)
 - \frac \d{\d t} \phi(y) + O(|z|^2) + O(h^N) .
\end{align}
The first term is $O(h^2)$ because $\dot y_\perp=\dot y_1+\dot y_{-1}$ and its derivatives are $O(\eps)$ by Theorem~\ref{thm:mfe-boris}. Since $\widehat B_0$ is a skew-symmetric matrix, the first term is again the time derivative of an expression 
in which the appearing
second and higher derivatives of $y$ can be substituted as functions of
$(y,\dot y)$;  cf.~\cite{hairer18ebo}. The same holds true for the second term, as is shown in the proof of Proposition 6.2 
of~\cite{hairer17smm}.

We have thus found a function $H_h(x,v)$ with the properties that uniformly for all $x$ in a bounded domain and all bounded $v$ with
$v_\perp=O(\eps)$ we have
\begin{align}\label{Hh-H}
H_h(x,v)&=H(x,v)+O(h^2) ,
\\
\label{ddt-Hh}
\frac{\d}{\d t} H_h(y(t),\dot y(t)) &= O(|z(t)|^2)+ O(h^N).
\end{align}
We now consider the equation for $z$. With the starting velocity \eqref{v-init-mod} we have $|z(0)| \le c_0 h^4$ for some constant $c_0$; see part (d) of Theorem~\ref{thm:mfe-boris}. The differential equation for $z_\perp=z_1+z_{-1}$ can be written as
$$
\dot z_\perp = \frac{4\eps}{h^2} z_\perp \times B_0 + O(\eps |z_\perp|) + O(\eps h^N).
$$
Multiplying this equation with $2 (z_\perp)^\top$ and noting that $2 (z_\perp)^\top\dot z_\perp=  (\d/\d t) |z_\perp|^2$, we obtain
$$
\frac \d{\d t} |z_\perp|^2 \le C\eps |z_\perp|^2 + O(\eps h^{N}),
$$
which shows that $|z_\perp(t)| \le \e^{\widetilde c\eps t} |z_\perp(0)|+ O(t \eps h^{N})$. Moreover, from the proof of Theorem~\ref{thm:mfe-boris} we have $|z_0(t)|\le C h^2 |z_\perp(t)|$.
Patching many short time intervals of length 1 together as in part (b) of the proof of Theorem~\ref{thm:drift}, we find that on each of these intervals up to time  $c\,\eps^{-1}$
(but not on longer time intervals $\eps^{-\alpha}$ with $\alpha>1$ because of the $\e^{\widetilde c\eps t}$ exponential growth of our bound of $z_\perp$),  we can apply Theorem~\ref{thm:mfe-boris} and the oscillatory component $z$ on the interval remains of size $O(h^4)$. By \eqref{ddt-Hh} we thus have
$$
H_h(y(t),\dot y(t)) = H_h(y(0),\dot y(0)) + O(t h^8).
$$
(Different to Section~6, we now do not put a superscript on $y$ and $z$ to designate the interval of length 1 in which $t$ lies).
Together with 
$$
H(x^n,v^n)= H_h(x^n,v^n)+O(h^2) = H_h (y(t_n),\dot y(t_n)) + O(h^2),
$$
this yields the stated result for the energy. 

The long-term smallness of the magnetic moment follows from \eqref{I-vperp} and the relation
$v_\perp^n = \dot y_\perp(t_n) - (-1)^n \dot z_\perp(t_n)+ O(\eps h^2)$. This yields $v_\perp^n =O(\eps)$ by the differential equations in part (c) of Theorem~\ref{thm:mfe-boris} for $y_{\pm1}$ and $z_{\pm 1}$, which contain a factor $\eps$ on the right-hand side. These functions are again patched together over many short intervals as is done in the proof of Theorem~\ref{thm:drift}.
\qed
\end{proof}

\subsection{Time scale $\eps^{-N}$ for $N>1$: Filtered variational integrator}
We have the following result on the long-term near-conservation of magnetic moment and energy by the filtered variational integrator with non-resonant large step sizes \eqref{hh-fvi} with \eqref{nonres}.

\begin{theorem} \label{them:drift-boris} Let $M > N$ be arbitrary positive integers. Under the assumptions of Theorem~\ref{thm:mfe-fvi} (in particular \eqref{hh-fvi}--\eqref{nonres} and an initial velocity bounded independently of $\eps$), and provided that the numerical positions $x^n$ of the filtered variational integrator
stay in a compact set $K$  for $0\le t \le c\,\eps^{-N}$
(with $K$  and $c$ independent of $\eps$ and $h$), the magnetic moment and the total energy along the numerical solution $(x^n,v^n)$ remain almost conserved over such long times:
$$
\begin{aligned}
| \,I(x^n,v^n)-I(x^0,v^0) \,| & \le C h
\\
| H(x^n,v^n)-H(x^0,v^0) |  & \le C h
\end{aligned}
\qquad \text{for}\quad 0\le t \le c\,\min(h^{-M},\eps^{-N}).
$$
The constant $C$ is independent of $\varepsilon$ and $h$ and $n$ with $0\leq nh \leq c/\eps$, but depends on the initial velocity bound, on bounds of derivatives of $B_1$ and $E$ on the compact set $K$,  on $c$, and on the choice of $M$ and $N$. 
\end{theorem}

\begin{proof} The proof uses arguments that are very similar to the proofs of Theorems 2.2 and 2.3 of \cite{hairer20lta} on the long-term near-conservation properties of the standard variational integrator for step sizes $h\le c\eps$. We therefore only indicate the main steps in the proof, which are marked as items (i)-(iv) below.

 To simplify the expressions for the remainder terms, we assume in the following the mild stepsize restriction $h^m\le \eps$ for some fixed $m> 2$ and we choose $M\ge mN$. This is only done for ease of presentation
and allows us to cover the time scale $\eps^{-N}$. Without this assumption we arrive at the stated time scale $\min(h^{-M},\eps^{-N})$.

(i) ({\it Lagrangian structure of the modulation equations}; cf.~\cite[(5.23)]{hairer20lta})
Over a time interval of length~$1$
we consider the modulation functions $z^k(t)$ of Theorem~\ref{thm:mfe-fvi} multiplied with the corresponding highly oscillatory exponentials:
$$
\text{
$y^k(t)=z^k(t)\mathrm e^{\mathrm{i}kt/\varepsilon}$ for $|k|\leq N$ and $y^k(t)=0$ for $|k|> N$. 
}
$$
We write $\mathbf{y}=(y^k)_{k\in\mathbb{Z}}$ and define the extended potentials
\[
\begin{aligned}
\mathcal{U}(\mathbf{y})&=\sum_{0\leq m\leq N \atop s(\alpha)=0}\frac{1}{m!}\,\phi^{(m)}(y^0)\mathbf{y}^\alpha\\
\mathcal{{A}}(\mathbf{y})&=\left({\mathcal{A}}_k(\mathbf{y})\right)_{k\in\mathbb{Z}}=\left(\sum_{0\leq m\leq N \atop s(\alpha)=k}\frac{1}{m!}{A}^{(m)}(y^0)\mathbf{y}^\alpha\right)_{k\in\mathbb{Z}},
\end{aligned}
\]
where the sums are taken over all multi-indices $\alpha = (\alpha^1,\dots,\alpha^m)$ with $\alpha^j\in \mathbb{Z}\setminus\{ 0 \}$ with prescribed sum $s(\alpha)=\alpha_1+\ldots+\alpha_m$, and where we use the notation 
$\phi^{(m)}(y^0)\mathbf{y}^\alpha= \phi^{(m)}(y^0)(y^{\alpha^1},\dots,y^{\alpha^m})$ and analogously for ${A}^{(m)}(y^0)\mathbf{y}^\alpha$. The terms for $m=0$ are to be interpreted as $\phi(y^0)$ and $A(y^0)$.

The system of modulation equations of the filtered variational integrator can then be written, up to $O(\eps^{N})$, as the discrete Euler-Lagrange equations corresponding to the discrete Lagrangian
\begin{align*}
\mathcal{L}_h(\mathbf{y}^n,\mathbf{y}^{n+1}) &= \frac h2(\mathbf{v}^{n+1/2})^* (\Psi^{-1}\otimes I)\mathbf{v}^{n+1/2} 
\\
&\ + 
\frac h2\bigl(\mathcal{{A}}(\mathbf{y}^n)+\mathcal{{A}}(\mathbf{y}^{n+1})\bigr)^* \mathbf{v}^{n+1/2} -
\frac h2\bigl(\mathcal{{U}}(\mathbf{y}^n)+\mathcal{{U}}(\mathbf{y}^{n+1})\bigr)
\end{align*}
with $\mathbf{v}^{n+1/2}=(\mathbf{y}^{n+1} - \mathbf{y}^{n})/h$, which differs from that of the standard variational integrator only by the modified kinetic energy term with $\Psi^{-1}$. We thus have
\begin{equation}\label{eq:modulation}
\Psi^{-1}\delta_h^2y^k=\sum_{j\in\mathbb{Z}}\left(\frac{\partial{ \mathcal{A}}_j}{\partial y^k}(\mathbf y)\right)^*\delta_{2h}y^j-\delta_{2h}{\mathcal{A}}_k(\mathbf y)-\left(\frac{\partial \mathcal{U}}{\partial y^k}(\mathbf y)\right)^*+O(\varepsilon^{N}),
\end{equation}
where $\delta_{2h} f(t) =(f(t+h)-f(t-h))/(2h)$ and $\delta_h^2f(t)=(f(t+h)-2f(t)+f(t-h))/h^2$ denote the first-order and second-order symmetric difference quotients, respectively.

(ii) ({\it Almost-invariant close to the magnetic moment}; cf.~\cite[Theorem 5.2]{hairer20lta})
With the group action $S(\lambda)\mathbf{y}=(e^{\mathrm{i}k\lambda}y^k)_{k\in\mathbb{Z}}$ (for $\lambda\in\real$), we have
\[
\mathcal{U}(S(\lambda)\mathbf{y})=\mathcal{U}(\mathbf{y}),\quad {\mathcal{A}}(S(\lambda)\mathbf{y})=S(\lambda){\mathcal{A}}(\mathbf{y}) \quad \text{for all} ~ \lambda.
\]
Differentiation with respect to $\lambda$ (at $\lambda=0$) yields
\[
\begin{aligned}
    \sum_{k\in\mathbb{Z}}\mathrm{i}k\frac{\partial \mathcal{U}}{\partial y^k}(\mathbf{y})y^k&=0\\
    \sum_{j\in\mathbb{Z}}\mathrm{i}j\frac{\partial {\mathcal{A}}_k}{\partial y^j}(\mathbf{y})y^j&=\mathrm{i}k{\mathcal{A}}_k(\mathbf{y}) \quad \text{for} \quad k\in\mathbb{Z}.
\end{aligned}
\]
Multiplying \eqref{eq:modulation} with $-\iu k (y^k)^*$, summing over $k$ and using these relations yields that the function
\[
\begin{aligned}
\mathcal{I}_h[\mathbf y](t)=&\ -\frac{\mathrm i}{\varepsilon h}\sum_k k y^k(t)^*\Psi^{-1}y^k(t+h)
\\
&\ +\frac{\mathrm i}{2\varepsilon}\sum_k k\left({\mathcal{A}}_k(\mathbf y(t))^*y^k(t+h)-y^k(t)^*{\mathcal{A}}_k(\mathbf{y}(t+h))\right)
\end{aligned}
\]
satisfies
\begin{equation}\label{Ih}
\mathcal{I}_h[\mathbf y](t)-\mathcal{I}_h[\mathbf y](t-h)=O(h\varepsilon^{N+1})
\end{equation}
and is thus an almost-invariant of the modulation system.
Using the bounds of the modulation functions, we find that
\begin{align*}
\mathcal{I}_h[\mathbf y](t) =&\ -\frac{\mathrm i}{\varepsilon h}\left(y^1(t)^*\Psi^{-1}  y^1(t+h)-y^{-1}(t)^*\Psi^{-1}y^{-1}(t+h)\right) 
\\
&\ +
\frac{\mathrm i}{2\varepsilon^2}\left(y^1(t)^*(y^1(t+h)\times B_0)-y^{-1}(t)^*(y^{-1}(t+h)\times B_0)\right) + O(\eps).
\end{align*}
Here, a calculation shows that the first term equals $(1+\cos(h/\varepsilon))|z^1_1|^2/\eps^2+O(h)$, and the second term equals
$-\cos(h/\eps)|z_1^1|^2/\eps^2+ O(h)$. So we obtain
$$
\mathcal{I}_h[\mathbf y](t)=\frac{1}{\varepsilon^2}|z_1^1(t)|^2+O(h).
$$
On the other hand, since 
\[
\begin{aligned}
v^n=&\ \Phi\,\frac{x^{n+1}-x^{n-1}}{2h}+O(\eps)\\
=&\left(I+(1-{\rm sinc}(h/\eps)^{-1})\hat{B}_0^2\right)\frac{x^{n+1}-x^{n-1}}{2h}+O(\eps)\\
=&\ \dot{z}^0(t)+\frac{{\rm i}}{\varepsilon}\left(z_1^1(t)\e^{{\rm{i}}t/\varepsilon}-z_{-1}^{-1}(t)\e^{{-\rm{i}}t/\varepsilon}\right)+O(h)
\quad\text{ at }t=nh
\end{aligned}
\]
and $\dot{z}^0(t)\times B_0=O(\eps)$,
we find that
$$
I(x^n,v^n)=
\tfrac12 |v^n\times B_0|^2 +O(\eps) = \frac{1}{\varepsilon^2}|z_1^1(t)|^2+O(h).
$$
So we obtain that the magnetic moment along the numerical solution is $O(h)$-close to the almost-invariant:
\begin{equation}\label{mag-mom}
 I(x^n,v^n)=\mathcal{I}_h[\mathbf y](nh) +O(h).
\end{equation}

(iii) ({\it Almost-invariant close to the total energy}; cf.~\cite[Theorem 5.3]{hairer20lta})
Multiplying \eqref{eq:modulation} with $ (\dot y^k)^*$ and summing over $k$ gives
\begin{align}\nonumber
\sum_k(\dot{y}^k)^*\Psi^{-1}\delta_h^2y^k
 &\ -\sum_k\left(\frac{\mathrm{d}}{\mathrm{d}t}\mathcal{A}_k(\mathbf y)^{*}\delta_{2h}y^k
  -(\dot{y}^k)^{*}\delta_{2h}\mathcal{A}_k(\mathbf y)\right)+\frac{\mathrm{d}}{\mathrm{d}t}\mathcal{U}(\mathbf y)
\\
&\hspace{6cm}=O(\varepsilon^{N}).
\label{eq:energyproof1}
\end{align}
The arguments in the proof of Theorem~5.3 in \cite{hairer20lta} 
show that each of the three terms on the left-hand side is a total differential up to $O(\eps^{N})$. So there exists a function 
\[
\mathcal{H}_h[\mathbf y](t)=\mathcal{K}_h[\mathbf y](t)
+\mathcal{M}_h[\mathbf y](t)+\mathcal{U}[\mathbf y](t),
\]
where the time derivatives of the three terms on the right-hand side equal the three corresponding terms on the left-hand side of (\ref{eq:energyproof1}), and we have
\[
\frac{\mathrm d}{\mathrm d t}\mathcal{H}_h[\mathbf y](t)=O(\varepsilon^{N}).
\]
We now determine the dominant part of $\mathcal{H}_h[\mathbf y]$.
We find
\[
\begin{aligned}
\mathcal{K}_h[\mathbf y]&=\tfrac{1}{2}(\dot{z}^0)^*\Psi^{-1}\dot{z}^0+\frac{2(z^1)^*\Psi^{-1}z^1}{h^2}\left(\frac{h}{\eps}\sin(h/\eps)-2\sin^2(h/2\eps)\right)+O(h)\\
&=\tfrac{1}{2}|\dot{z}^0|^2+\frac{2|z^1_1|^2}{{\rm tanc}(h/2\varepsilon)h^2}\left(\frac{h}{\eps}\sin(h/\eps)-2\sin^2(h/2\eps)\right)+O(h)\\
&=\tfrac{1}{2}|\dot{z}^0|^2+\frac{|z^1_1|^2}{\eps^2}(1+\cos(h/\eps))+O(h)
\\[2mm]
\mathcal{M}_h[\mathbf y] &= -\cos(h/\eps) \frac{|z^1_1|^2}{\eps^2}+ O(h)
\\[2mm]
\mathcal{U}[\mathbf y] &= \phi(z^0)+O(h).
\end{aligned}
\]
Thus we have
\begin{equation} \label{Hh}
\mathcal{H}_h[\mathbf y](t)=\tfrac{1}{2}|\dot{z}^0(t)|^2+\frac{|z_1^1(t)|^2}{\varepsilon^2}+\phi(z^0(t))+O(h).
\end{equation}
On the other hand, from the formula for $v^n$ in (ii) we have, at $t=nh$,
\[
\tfrac{1}{2}|v^n|^2=\tfrac{1}{2}|\dot{z}^0(t)|^2+\frac{|z_1^1(t)|^2}{\varepsilon^2}+O(h).
\]
The energy along the numerical solution is therefore
\[
H(x^n,v^n)=\tfrac{1}{2}|v^n|^2+\phi(x^n)=\tfrac{1}{2}|\dot{z}^0(t)|^2+\frac{|z_1^1(t)|^2}{\varepsilon^2}+\phi(z^0(t))+O(h).
\]
and hence we have
\[
H(x^n,v^n)=\mathcal{H}_h[\mathbf y](t) + O(h).
\]

(iv) ({\it From short to long time intervals}; cf.~\cite[Section 4.5]{hairer20lta}, \cite[Section XIII.7]{hairer06gni}).
The stated long-time near-conservation results are
 now  obtained  by  patching together  the  short-time  near-conservation  results  of  (ii) and (iii) over many intervals of length 1, via an often-used argument that involves the uniqueness up to $O(\eps^{N+1})$ of the modulation functions.
 \qed
\end{proof}

\medskip\noindent{\bf Numerical experiment.} We illustrate the energy behaviour of the numerical methods for the magnetic field
\[
B(x)=
\frac{1}{\varepsilon}
\begin{pmatrix}
1\\0\\0.5
\end{pmatrix}+
\begin{pmatrix}
x_2-x_3\\x_1+x_3\\x_2-x_1,
\end{pmatrix}
\]  
and the scalar potential
$
\phi(x)=x_1^3-x_2^3+\frac{1}{5}x_1^4+x_2^4+x_3^4.
$ 
We take the initial values 
$
x(0)=(0,1,0.1)^\top,\ v(0)=(0.09,0.05,0.2)^\top.
$ 

We apply the three numerical integrators of Section 3 with $\eps=10^{-4}$, step size $h=10^{-2}$, and final time $T=10^7$. Figure \ref{fig:energy} shows the energy error $H(x_n,v_n)-H(x_0,v_0)$ along the numerical solutions of the Boris algorithm, the standard variational integrator and the filtered variational integrator, taking the initial values $x(0),v(0)$ as starting values for all three methods. 

The errors of the Boris algorithm and the variational integrator (top and centre picture) appear to behave randomly. Running several trajectories corresponding to random perturbations of the initial data of magnitude $10^{-14}$ showed energy errors that look like random walks with a deviation of magnitude 10 for $t\le 10^{6}$. For larger times, some of the trajectories showed blow-up behaviour.

In contrast, the energy error of the filtered variational integrator oscillates with a small amplitude without drift 
(bottom picture of Figure \ref{fig:energy}).
The error $I(x_n,v_n)-I(x_0,v_0)$ of the magnetic moment along the numerical solution of the filtered variational integrator has a very similar behaviour (not shown here).

If we apply the Boris algorithm and the standard variational integrator with {\it modified initial values} \eqref{v-init-mod}, then the magnetic moment remains small over very long time, oscillating between $0$ and approximately $2\cdot 10^{-6}$ over the whole time interval. In this case of modified initial velocity, we observe very good near-conservation of energy for the variational integrator while there is a linear drift for the Boris algorithm; see Figure~\ref{fig:energy2}. 

\begin{figure}[h]
\centerline{\includegraphics[scale=0.48]{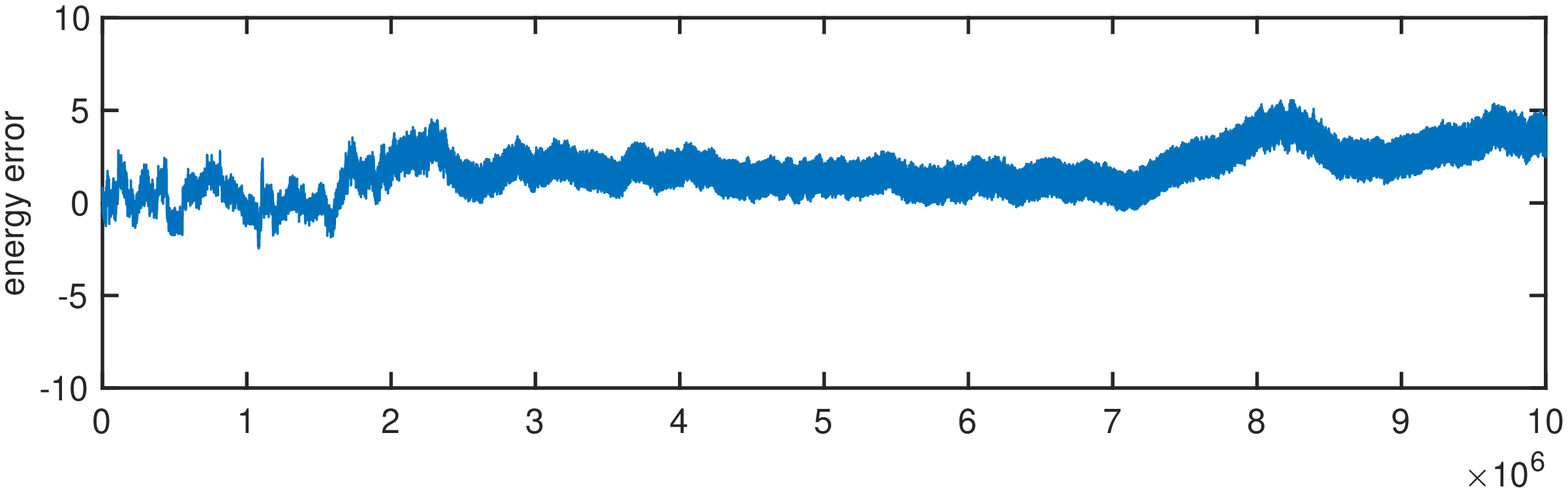}}
\centerline{\includegraphics[scale=0.48]{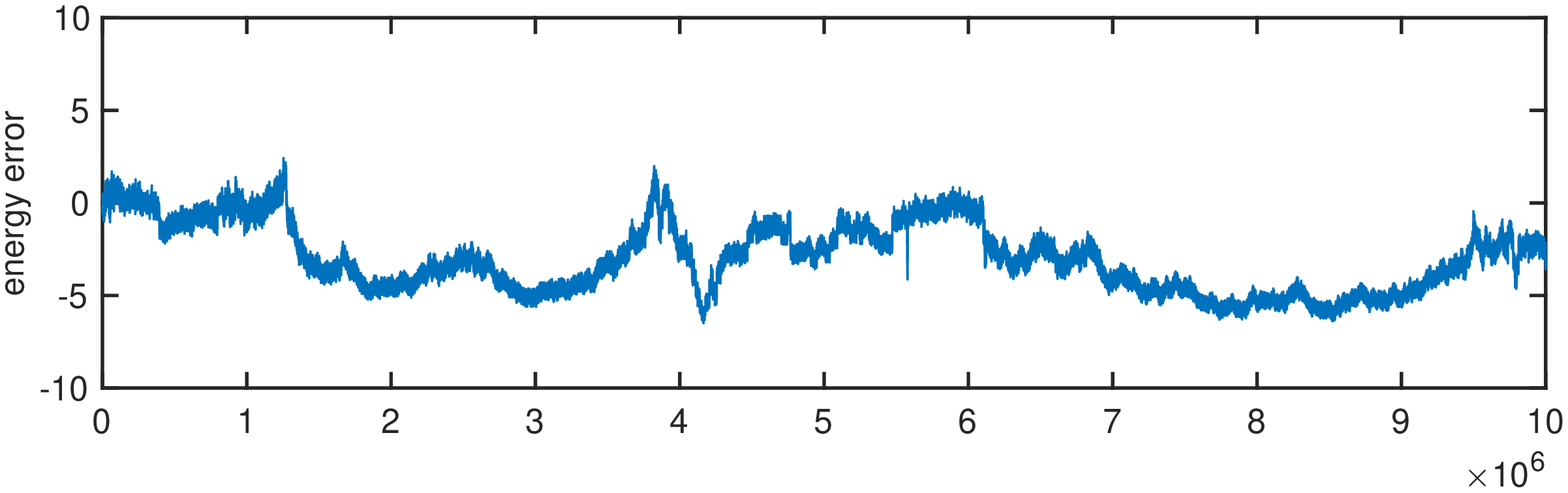}}
\centerline{\includegraphics[scale=0.48]{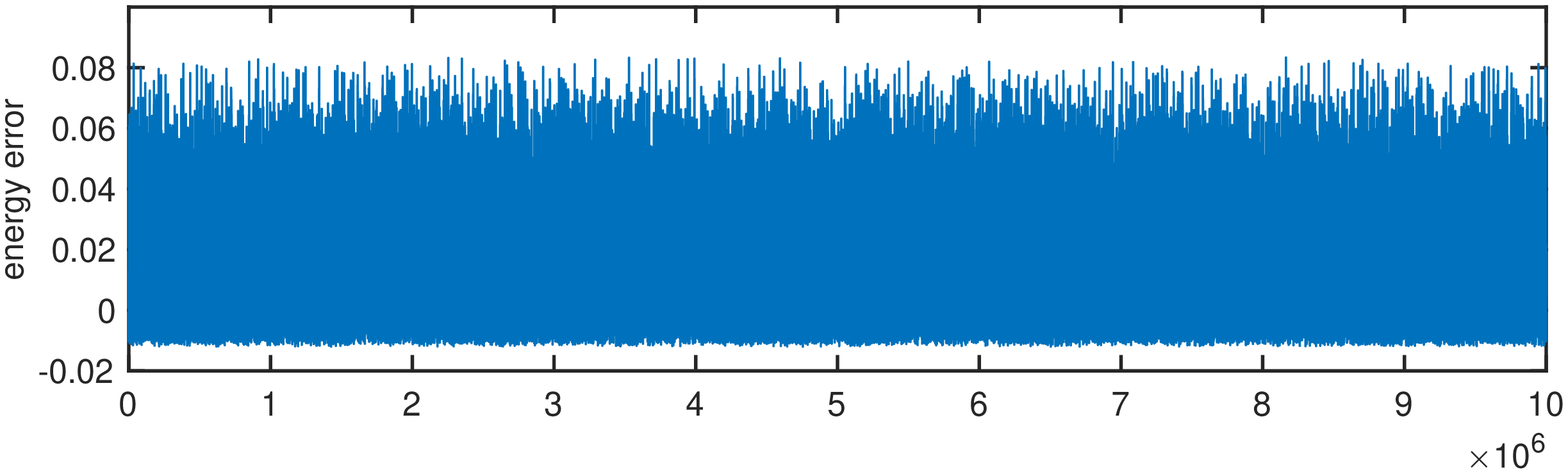}}
\caption{Energy error $H(x_n,v_n)-H(x_0,v_0)$ along the numerical solutions of the Boris algorithm (top), of the standard variational integrator (centre) and of the filtered variational integrator (bottom), obtained with $\eps=10^{-4}$ and $h=10^{-2}$.}\label{fig:energy}
\end{figure}


\begin{figure}[h]
\centerline{\includegraphics[scale=0.48]{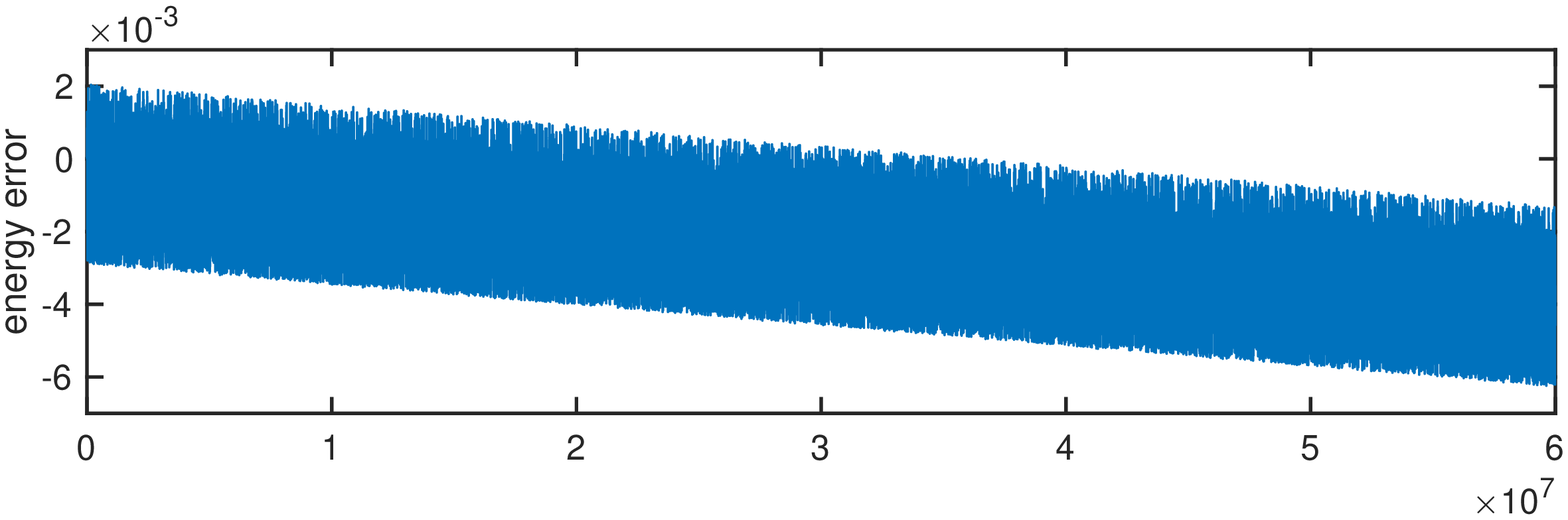}}
\centerline{\includegraphics[scale=0.48]{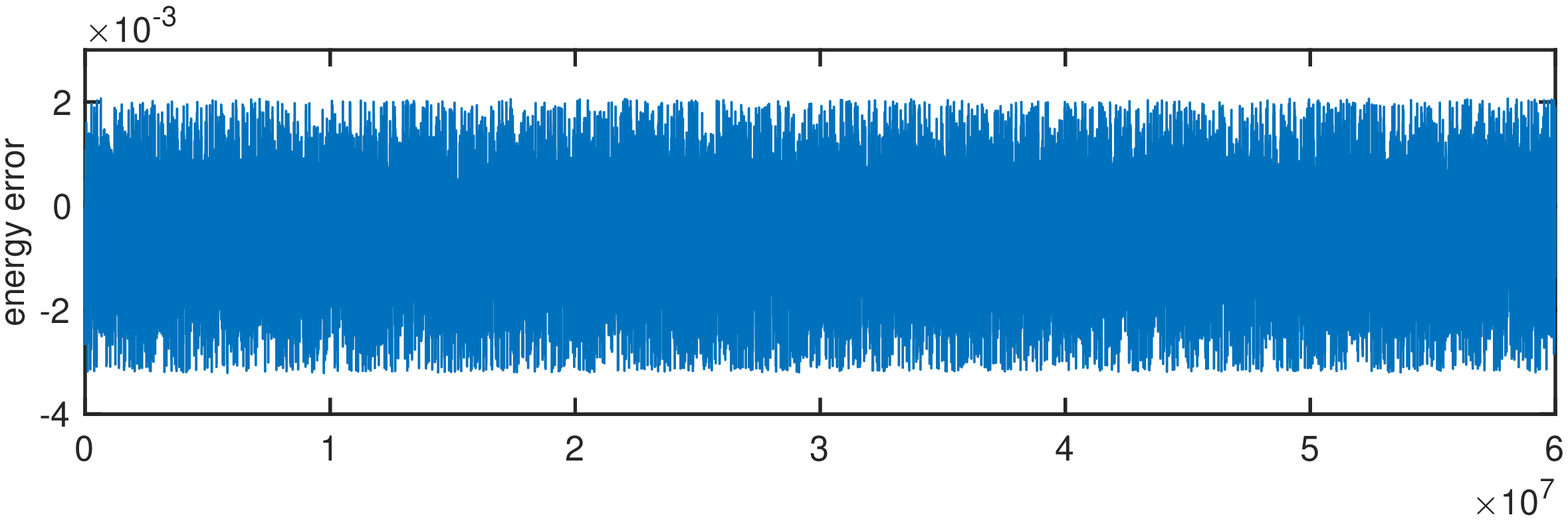}}
\caption{Energy error $H(x_n,v_n)-H(x_0,v_0)$ along the numerical solutions of the Boris algorithm (top) and of the standard variational integrator (bottom) obtained with modified initial values \eqref{v-init-mod} and step size $h=10^{-2}$, for $\eps=10^{-4}$.}\label{fig:energy2}
\end{figure}

\section*{Acknowledgement} 
This work was partially supported by the Swiss National Science Foundation, grant No. 200020\_192129,
and by the Deutsche Forschungsgemeinschaft (DFG, German Research Foundation) -- Project-ID258734477 -- SFB 1173.  The work by Yanyan Shi was done at the University of T\"ubingen during her one-year research stay, which was funded by a scholarship provided by the University of the Chinese Academy of Sciences (UCAS).

\renewcommand{\refname}{\normalsize\bf References}
\small
\bibliographystyle{acm}
\bibliography{HLW}

\end{document}